\documentclass[secthm,seceqn,amsthm,ussrhead,12pt]{amsart}
\usepackage[utf8]{inputenc}
\usepackage[english]{babel}

\usepackage{amssymb,amsmath,amsthm,amsfonts,xcolor,enumerate,hyperref,comment,longtable,cleveref}
\usepackage{soul}
\usepackage{times}
\usepackage{cite}
\usepackage{pdflscape}
\usepackage{ulem}
\usepackage[mathcal]{euscript}
\usepackage{tikz}
\usepackage{hyperref}
\usepackage{cancel}
\usepackage{stmaryrd}
\usepackage{marginnote}

\usetikzlibrary{arrows}

\newtheorem{theorem}{Theorem}
\newtheorem{lemma}[theorem]{Lemma}

\newtheorem{remark}[theorem]{Remark}

\newtheorem{definition}[theorem]{Definition}
\newtheorem{corollary}[theorem]{Corollary}

\newtheorem*{theoremA}{Theorem A}

\newtheorem*{theoremB}{Theorem B}

\usepackage{stmaryrd}
\usepackage{xcolor}

\begin{document}

 \medskip
 
   \medskip

 \medskip

\noindent{\Large
The algebraic and geometric classification of   Zinbiel  algebras}
 
 \hspace*{50mm} 
 
 \medskip
 
   \medskip

 \medskip 
 
   {\bf  
   Mar\'ia Alejandra Alvarez,
   Renato Fehlberg J\'{u}nior
  $\&$
    Ivan Kaygorodov \\

    \medskip
}

 \medskip

 \medskip
 
\ 

\noindent{\bf Abstract}:
{\it This paper is devoted to the complete algebraic and geometric classification of complex  $5$-dimensional Zinbiel  algebras.
In particular, 
we proved that
the variety of complex $5$-dimensional  Zinbiel algebras has dimension $24$, it is defined by $16$ irreducible components and it has $11$ rigid algebras.
 }
 
\medskip

 \medskip

 \medskip

 \medskip
 
\hspace*{50mm} 
  

\medskip

 \hspace*{50mm} 
 

 \medskip

 \medskip

 \medskip

  \hspace{100mm} 

\section*{Introduction}
The algebraic classification (up to isomorphism) of algebras of dimension $n$ from a certain variety
defined by a certain family of polynomial identities is a classic problem in the theory of non-associative algebras.
There are many results related to the algebraic classification of small-dimensional algebras in the varieties of
Jordan, Lie, Leibniz, Zinbiel and many other algebras \cite{             degr3,      degr2,        kkp20}.
 Geometric properties of a variety of algebras defined by a family of polynomial identities have been an object of study since 1970's (see, \cite{wolf2, wolf1, chouhy,   fkkv22,  alesl,  aleis, aleis2,   gabriel,      cibils,  shaf, GRH, GRH2, ale3,      kppv, S90}). 
 Gabriel described the irreducible components of the variety of $4$-dimensional unital associative algebras~\cite{gabriel}.  
 Cibils considered rigid associative algebras with $2$-step nilpotent radical \cite{cibils}.
 Grunewald and O'Halloran computed the degenerations for the variety of $5$-dimensional nilpotent Lie algebras~\cite{GRH}. 
 All irreducible components of  $2$-step nilpotent  commutative  and anticommutative  algebras have been described in \cite{shaf,ikp20}.
Chouhy  proved that  in the case of finite-dimensional associative algebras,
 the $N$-Koszul property is preserved under the degeneration relation~\cite{chouhy}.
Degenerations have also been used to study a level of complexity of an algebra~\cite{wolf1,wolf2,gorb93}.
 The study of degenerations of algebras is very rich and closely related to deformation theory, in the sense of Gerstenhaber \cite{ger63}.

\newpage 
Loday introduced a class of symmetric operads generated by one bilinear operation subject to one relation making each left-normed product of three elements equal to a linear combination of right-normed products:
    $(a_1a_2)a_3=\sum\limits_{\sigma\in \mathbb{S}_3} x_{\sigma} a_{\sigma(1)}(a_{\sigma(2)}a_{\sigma(3)});$
such an operad is called a parametrized one-relation operad. For a particular choice of parameters $\{x_{\sigma}\}$, this operad is said to be regular if each of its components is the regular representation of the symmetric group; equivalently, the corresponding free algebra on a vector space $V$ is, as a graded vector space, isomorphic to the tensor algebra of $V$. 
Bremner and Dotsenko classified, over an algebraically closed field of characteristic zero, all regular parametrized one-relation operads. In fact, they proved that each such operad is isomorphic to one of the following five operads: 
the left-nilpotent operad, the associative operad, the Leibniz operad, the Zinbiel operad, and the Poisson operad \cite{bredo}. 
An algebra $\bf A$ is called a (left) {\it Zinbiel algebra} if it satisfies the identity 
$(xy)z=x(yz+zy).$
	Zinbiel algebras were introduced by Loday in \cite{loday}.
Under the Koszul duality, the operad of Zinbiel algebras is dual to the operad of Leibniz algebras. 
Zinbiel algebras are also known as pre-commutative algebras \cite{ pasha}
and chronological algebras \cite{Kawski}.
A Zinbiel algebra is equivalent to a commutative dendriform algebra \cite{comdend}.
It plays an important role in the definition of pre-Gerstenhaber algebras
\cite{Aloulou}.
The variety of Zinbiel algebras is a proper subvariety in the variety of right commutative algebras.
Each  Zinbiel algebra with the commutator multiplication gives a Tortkara algebra \cite{dzhuma}, which has   sprung up in unexpected areas of mathematics \cite{tortnew1,tortnew2}.
Recently, the notion of matching Zinbiel algebras was introduced in \cite{matchzin}.
Recently, Zinbiel algebras also  appeared in a study of rack cohomology \cite{rack}, 
number theory \cite{Chapoton21} and 
in  a  construction of a Cartesian differential category  \cite{ip21}.
In recent years, there has been a strong interest in the study of Zinbiel algebras in the algebraic and the operad context  
\cite{ cam13,mukh, dok,  anau,34,  abp,matchzin, centr3zinb, dzhuma,    dzhuma5, dzhuma19, kppv, ualbay}. 

Free Zinbiel algebras were shown to
be precisely the shuffle product algebra \cite{34}.
Naurazbekova proved that, over a field of characteristic zero, free Zinbiel algebras are the free associative-commutative algebras (without unity) with respect to the symmetrization multiplication and their free generators are found; also she constructed examples of subalgebras of the two-generated free Zinbiel algebra that are free Zinbiel algebras of countable rank \cite{anau}.
Nilpotent algebras play an important role in the class of Zinbiel algebras.
So, 
Dzhumadildaev and  Tulenbaev proved that each complex finite dimensional Zinbiel algebra is nilpotent  \cite{dzhuma5};
Naurazbekova 
and Umirbaev proved that in characteristic zero any proper subvariety of the variety of Zinbiel algebras is nilpotent \cite{ualbay}.
Finite-dimensional Zinbiel algebras with a ``big'' nilpotency index are classified in \cite{adashev,cam13}.
Central extensions of three dimensional Zinbiel algebras were calculated in \cite{centr3zinb} and of filiform Zinbiel algebras in \cite{cam20}.
The full system of degenerations of complex four dimensional Zinbiel algebras is given in \cite{kppv}.

Our method for classifying nilpotent  Zinbiel  algebras is based on the calculation of central extensions of nilpotent algebras of smaller dimensions from the same variety. 
The algebraic study of central extensions of   algebras has been an important topic for years \cite{  klp20,hac16,  ss78}.
First, Skjelbred and Sund used central extensions of Lie algebras to obtain a classification of nilpotent Lie algebras  \cite{ss78}.
Note that the Skjelbred-Sund method of central extensions is an important tool in the classification of nilpotent algebras.
Using the same method,  
 small dimensional nilpotent 
(associative, 
 terminal  \cite{kkp20}, Jordan,
  Lie  \cite{degr3,degr2}, 
 anticommutative  algebras,
and some others) have been described. Our main results related to the algebraic classification of cited varieties are summarized below.

\begin{theoremA}
Up to isomorphism, there are infinitely many isomorphism classes of  
complex  $5$-dimensional non-split non-2-step nilpotent Zinbiel algebras, 
described explicitly  in  section \ref{secteoA} in terms of 
$6$ one-parameter families and 
$53$ additional isomorphism classes.

\end{theoremA}

 The degenerations between the (finite-dimensional) algebras from a certain variety $\mathfrak{V}$ defined by a set of identities have been actively studied in the past decade.
The description of all degenerations allows one to find the so-called rigid algebras and families of algebras, i.e. those whose orbit closures under the action of the general linear group form irreducible components of $\mathfrak{V}$
(with respect to the Zariski topology). 
We list here some works in which the rigid algebras of the varieties of
all $4$-dimensional Leibniz algebras,
all $4$-dimensional nilpotent terminal algebras \cite{kkp20},
all $4$-dimensional nilpotent commutative algebras,
all $6$-dimensional nilpotent binary Lie algebras,
all $6$-dimensional nilpotent anticommutative algebras 
have been found.
A full description of degenerations has been obtained  
for $2$-dimensional algebras, 
for $4$-dimensional Lie algebras,
for $4$-dimensional Zinbiel and  $4$-dimensional nilpotent Leibniz algebras in \cite{kppv},
for $6$-dimensional nilpotent Lie algebras in \cite{S90,GRH},  
for $8$-dimensional $2$-step nilpotent anticommutative algebras \cite{ale3},
and for $(n+1)$-dimensional $n$-Lie algebras.
Our main results related to the geometric classification of cited varieties are summarized below. 
\begin{theoremB}
The variety of complex $5$-dimensional   Zinbiel algebras has 
dimension $24$  and it has 
$16$ irreducible components (in particular, there are only $11$ rigid algebras in this variety). 

\end{theoremB}

\subsection{Symmetric Zinbiel algebras}
Similar to the Leibniz case, obvious that there are left and right Zinbiel algebras.
Hence, the notation of symmetric Zinbiel algebra should be introduced by the similar way 
with symmetric Leibniz algebras
(about symmetric Leibniz algebras see \cite{bsaid} and references therein).
	An algebra $\bf A$ is called a {\it symmetric Zinbiel algebra} if it satisfies the identities 
\begin{center}
    $(xy)z=x(yz+zy), \, x(yz)=(xy+yx)z.$	
\end{center}	
The operad of symmetric Zinbiel algebras is dual to the operad of symmetric Leibniz algebras.
The variety of symmetric Zinbiel algebras is a proper subvariety in the variety of bicommutative algebras.
The following Lemma can be obtained by some tedious calculations.

\begin{lemma}
Let $\mathcal S$ be a symmetric Zinbiel algebra.
Then $\mathcal S$ is a $3$-step nilpotent algebra 
and it satisfies the following two identities
\begin{center}
$(xy)z=-y(zx)$ and $(xy)z=-z(yx).$ 
\end{center}
\end{lemma}

\begin{corollary}
Each $n$-dimensional ($n<6$) symmetric Zinbiel algebra is $2$-step nilpotent.
There is a non-$2$-step nilpotent $6$-dimensional symmetric Zinbiel algebra.
\end{corollary}

\begin{proof}
The first part of the present statement follows from Theorem A, 
because there are no symmetric algebras in the classification of Theorem A.
The required $6$-dimensional symmetric Zinbiel algebra is given below:
\begin{longtable}{lllll}
$e_1 e_2=e_3$&  $e_2e_1=e_4$& $e_2e_2=e_5$ & $e_1e_5=e_6$ &
$e_5e_1=-e_6$ \\
$e_2e_4=-2 e_6$ & $e_4e_2=-e_6 $& $e_2e_3=e_6$& 
$e_3e_2=2e_6$
\end{longtable}

\end{proof}

\section{The algebraic classification of Zinbiel  algebras}
\subsection{Method of classification of nilpotent algebras}
Throughout this paper, we use the notations and methods well written in \cite{klp20,  hac16},
which we have adapted for the Zinbiel case with some modifications.
Further in this section we give some important definitions.

Let $({\bf A}, \cdot)$ be a Zinbiel algebra over  $\mathbb{C}$  and let $\mathbb V$ be a vector space over $\mathbb{C}$. Then the $\mathbb{C}$-linear space ${\rm Z}^2\left(
\bf A,\mathbb V \right) $ is defined as the set of all  bilinear maps $%
\theta :{\bf A} \times {\bf A} \longrightarrow {\mathbb V},$
such that \[\theta(xy,z)=\theta(x,yz+zy).\]

These elements will be called {\it cocycles}. For a
linear map $f$ from $\bf A$ to  $\mathbb V$, if we define $\delta f\colon {\bf A} \times
{\bf A} \longrightarrow {\mathbb V}$ by $\delta f  (x,y ) =f(xy )$, then $\delta f\in {\rm Z^{2}}\left( {\bf A},{\mathbb V} \right) $. We define ${\rm B^{2}}\left({\bf A},{\mathbb V}\right) =\left\{ \theta =\delta f\ : f\in {\rm Hom}\left( {\bf A},{\mathbb V}\right) \right\} $.
We define the {\it second cohomology space} ${\rm H^{2}}\left( {\bf A},{\mathbb V}\right) $ as the quotient space ${\rm Z^{2}}
\left( {\bf A},{\mathbb V}\right) \big/{\rm B^{2}}\left( {\bf A},{\mathbb V}\right) $.

Let $\operatorname{Aut}({\bf A}) $ be the automorphism group of  ${\bf A} $ and let $\phi \in \operatorname{Aut}({\bf A})$. For $\theta \in
{\rm Z^{2}}\left( {\bf A},{\mathbb V}\right) $ define  the action of the group $\operatorname{Aut}({\bf A}) $ on ${\rm Z^{2}}\left( {\bf A},{\mathbb V}\right) $ by $\phi \theta (x,y)
=\theta \left( \phi \left( x\right) ,\phi \left( y\right) \right) $.  It is easy to verify that
 ${\rm B^{2}}\left( {\bf A},{\mathbb V}\right) $ is invariant under the action of $\operatorname{Aut}({\bf A}).$  
 So, we have an induced action of  $\operatorname{Aut}({\bf A})$  on ${\rm H^{2}}\left( {\bf A},{\mathbb V}\right)$.

Let $\bf A$ be a  Zinbiel  algebra of dimension $m$ over  $\mathbb C$ and ${\mathbb V}$ be a $\mathbb C$-vector
space of dimension $k$. For the bilinear map $\theta$, define on the linear space ${\bf A}_{\theta } = {\bf A}\oplus {\mathbb V}$ the
bilinear product `` $\left[ -,-\right] _{{\bf A}_{\theta }}$'' by $\left[ x+x^{\prime },y+y^{\prime }\right] _{{\bf A}_{\theta }}=
 xy +\theta(x,y) $ for all $x,y\in {\bf A},x^{\prime },y^{\prime }\in {\mathbb V}$.
The algebra ${\bf A}_{\theta }$ is called a $k$-{\it dimensional central extension} of ${\bf A}$ by ${\mathbb V}$. One can easily check that ${\bf A_{\theta}}$ is a  Zinbiel  
algebra if and only if $\theta \in {\rm Z^2}({\bf A}, {\mathbb V})$.

Call the
set $\operatorname{Ann}(\theta)=\left\{ x\in {\bf A}:\theta \left( x, {\bf A} \right)+ \theta \left({\bf A} ,x\right) =0\right\} $
the {\it annihilator} of $\theta $. We recall that the {\it annihilator} of an  algebra ${\bf A}$ is defined as
the ideal $\operatorname{Ann}(  {\bf A} ) =\left\{ x\in {\bf A}:  x{\bf A}+ {\bf A}x =0\right\}$. Observe
 that
$\operatorname{Ann}\left( {\bf A}_{\theta }\right) =(\operatorname{Ann}(\theta) \cap\operatorname{Ann}({\bf A}))
 \oplus {\mathbb V}$.

\

The following result shows that every algebra with a nonzero annihilator is a central extension of a smaller-dimensional algebra.

\begin{lemma}
Let ${\bf A}$ be an $n$-dimensional  Zinbiel  algebra such that $\dim (\operatorname{Ann}({\bf A}))=m\neq0$. Then there exists, up to isomorphism, an unique $(n-m)$-dimensional  Zinbiel   algebra ${\bf A}'$ and a bilinear map $\theta \in {\rm Z^2}({\bf A'}, {\mathbb V})$ with $\operatorname{Ann}({\bf A'})\cap\operatorname{Ann}(\theta)=0$, where $\mathbb V$ is a vector space of dimension m, such that ${\bf A} \cong {{\bf A}'}_{\theta}$ and
 ${\bf A}/\operatorname{Ann}({\bf A})\cong {\bf A}'$.
\end{lemma}

\begin{proof}
Let ${\bf A}'$ be a linear complement of $\operatorname{Ann}({\bf A})$ in ${\bf A}$. Define a linear map $P \colon {\bf A} \longrightarrow {\bf A}'$ by $P(x+v)=x$ for $x\in {\bf A}'$ and $v\in\operatorname{Ann}({\bf A})$, and define a multiplication on ${\bf A}'$ by $[x, y]_{{\bf A}'}=P(x y)$ for $x, y \in {\bf A}'$.
For $x, y \in {\bf A}$, we have
\[P(xy)=P((x-P(x)+P(x))(y- P(y)+P(y)))=P(P(x) P(y))=[P(x), P(y)]_{{\bf A}'}. \]

Since $P$ is a homomorphism, $P({\bf A})={\bf A}'$ is a  Zinbiel  algebra and
 ${\bf A}/\operatorname{Ann}({\bf A})\cong {\bf A}'$, which gives us the uniqueness. Now, define the map $\theta \colon {\bf A}' \times {\bf A}' \longrightarrow\operatorname{Ann}({\bf A})$ by $\theta(x,y)=xy- [x,y]_{{\bf A}'}$.
  Thus, ${\bf A}'_{\theta}$ is ${\bf A}$ and therefore $\theta \in {\rm Z^2}({\bf A'}, {\mathbb V})$ and $\operatorname{Ann}({\bf A'})\cap\operatorname{Ann}(\theta)=0$.
\end{proof}

\begin{definition}
Let ${\bf A}$ be an algebra and $I$ be a subspace of $\operatorname{Ann}({\bf A})$. If ${\bf A}={\bf A}_0 \oplus I$
then $I$ is called an {\it annihilator component} of ${\bf A}$.
A central extension of an algebra $\bf A$ without annihilator component is called a {\it non-split central extension}.
\end{definition}

Our task is to find all central extensions of an algebra $\bf A$ by
a space ${\mathbb V}$.  In order to solve the isomorphism problem we need to study the
action of $\operatorname{Aut}({\bf A})$ on ${\rm H^{2}}\left( {\bf A},{\mathbb V}
\right) $. To do that, let us fix a basis $\{e_{1},\ldots ,e_{s}\}$ of ${\mathbb V}$, and $
\theta \in {\rm Z^{2}}\left( {\bf A},{\mathbb V}\right) $. Then $\theta $ can be uniquely
written as $\theta \left( x,y\right) =
\displaystyle \sum_{i=1}^{s} \theta _{i}\left( x,y\right) e_{i}$, where $\theta _{i}\in
{\rm Z^{2}}\left( {\bf A},\mathbb C\right) $. Moreover, $\operatorname{Ann}(\theta)=\operatorname{Ann}(\theta _{1})\cap\operatorname{Ann}(\theta _{2})\cap\ldots \cap\operatorname{Ann}(\theta _{s})$. Furthermore, $\theta \in
{\rm B^{2}}\left( {\bf A},{\mathbb V}\right) $\ if and only if all $\theta _{i}\in {\rm B^{2}}\left( {\bf A},
\mathbb C\right) $.
It is not difficult to prove (see \cite[Lemma 13]{hac16}) that given a  Zinbiel  algebra ${\bf A}_{\theta}$, if we write as
above $\theta \left( x,y\right) = \displaystyle \sum_{i=1}^{s} \theta_{i}\left( x,y\right) e_{i}\in {\rm Z^{2}}\left( {\bf A},{\mathbb V}\right) $ and 
$\operatorname{Ann}(\theta)\cap \operatorname{Ann}\left( {\bf A}\right) =0$, then ${\bf A}_{\theta }$ has an
annihilator component if and only if $\left[ \theta _{1}\right] ,\left[
\theta _{2}\right] ,\ldots ,\left[ \theta _{s}\right] $ are linearly
dependent in ${\rm H^{2}}\left( {\bf A},\mathbb C\right) $.

\;

Let ${\mathbb V}$ be a finite-dimensional vector space over $\mathbb C$. The {\it Grassmannian} $G_{k}\left( {\mathbb V}\right) $ is the set of all $k$-dimensional
linear subspaces of $ {\mathbb V}$. Let $G_{s}\left( {\rm H^{2}}\left( {\bf A},\mathbb C\right) \right) $ be the Grassmannian of subspaces of dimension $s$ in
${\rm H^{2}}\left( {\bf A},\mathbb C\right) $. There is a natural action of $\operatorname{Aut}({\bf A})$ on $G_{s}\left( {\rm H^{2}}\left( {\bf A},\mathbb C\right) \right) $.
Let $\phi \in \operatorname{Aut}({\bf A})$. For $W=\left\langle
\left[ \theta _{1}\right] ,\left[ \theta _{2}\right] ,\dots,\left[ \theta _{s}
\right] \right\rangle \in G_{s}\left( {\rm H^{2}}\left( {\bf A},\mathbb C
\right) \right) $ define $\phi W=\left\langle \left[ \phi \theta _{1}\right]
,\left[ \phi \theta _{2}\right] ,\dots,\left[ \phi \theta _{s}\right]
\right\rangle $. We denote the orbit of $W\in G_{s}\left(
{\rm H^{2}}\left( {\bf A},\mathbb C\right) \right) $ under the action of $\operatorname{Aut}({\bf A})$ by $\operatorname{Orb}(W)$. Given
\[
W_{1}=\left\langle \left[ \theta _{1}\right] ,\left[ \theta _{2}\right] ,\dots,
\left[ \theta _{s}\right] \right\rangle ,W_{2}=\left\langle \left[ \vartheta
_{1}\right] ,\left[ \vartheta _{2}\right] ,\dots,\left[ \vartheta _{s}\right]
\right\rangle \in G_{s}\left( {\rm H^{2}}\left( {\bf A},\mathbb C\right)
\right),
\]
we easily have that if $W_{1}=W_{2}$, then $ \bigcap\limits_{i=1}^{s}\operatorname{Ann}(\theta _{i})\cap \operatorname{Ann}\left( {\bf A}\right) = \bigcap\limits_{i=1}^{s}
\operatorname{Ann}(\vartheta _{i})\cap\operatorname{Ann}( {\bf A}) $, and therefore we can introduce
the set
\[
{\bf T}_{s}({\bf A}) =\left\{ W=\left\langle \left[ \theta _{1}\right] ,
\left[ \theta _{2}\right] ,\dots,\left[ \theta _{s}\right] \right\rangle \in
G_{s}\left( {\rm H^{2}}\left( {\bf A},\mathbb C\right) \right) : \bigcap\limits_{i=1}^{s}\operatorname{Ann}(\theta _{i})\cap\operatorname{Ann}({\bf A}) =0\right\},
\]
which is stable under the action of $\operatorname{Aut}({\bf A})$.

\

Now, let ${\mathbb V}$ be an $s$-dimensional linear space and let us denote by
${\bf E}\left( {\bf A},{\mathbb V}\right) $ the set of all {\it non-split $s$-dimensional central extensions} of ${\bf A}$ by
${\mathbb V}$. By above, we can write
\[
{\bf E}\left( {\bf A},{\mathbb V}\right) =\left\{ {\bf A}_{\theta }:\theta \left( x,y\right) = \sum_{i=1}^{s}\theta _{i}\left( x,y\right) e_{i} \ \ \text{and} \ \ \left\langle \left[ \theta _{1}\right] ,\left[ \theta _{2}\right] ,\dots,
\left[ \theta _{s}\right] \right\rangle \in {\bf T}_{s}({\bf A}) \right\} .
\]
We also have the following result, which can be proved as in \cite[Lemma 17]{hac16}.

\begin{lemma}
 Let ${\bf A}_{\theta },{\bf A}_{\vartheta }\in {\bf E}\left( {\bf A},{\mathbb V}\right) $. Suppose that $\theta \left( x,y\right) =  \displaystyle \sum_{i=1}^{s}
\theta _{i}\left( x,y\right) e_{i}$ and $\vartheta \left( x,y\right) =
\displaystyle \sum_{i=1}^{s} \vartheta _{i}\left( x,y\right) e_{i}$.
Then the  Zinbiel  algebras ${\bf A}_{\theta }$ and ${\bf A}_{\vartheta } $ are isomorphic
if and only if
$$\operatorname{Orb}\left\langle \left[ \theta _{1}\right] ,
\left[ \theta _{2}\right] ,\dots,\left[ \theta _{s}\right] \right\rangle =
\operatorname{Orb}\left\langle \left[ \vartheta _{1}\right] ,\left[ \vartheta
_{2}\right] ,\dots,\left[ \vartheta _{s}\right] \right\rangle .$$
\end{lemma}

This shows that there exists a one-to-one correspondence between the set of $\operatorname{Aut}({\bf A})$-orbits on ${\bf T}_{s}\left( {\bf A}\right) $ and the set of
isomorphism classes of ${\bf E}\left( {\bf A},{\mathbb V}\right) $. Consequently we have a
procedure that allows us, given a  Zinbiel  algebra ${\bf A}'$ of
dimension $n-s$, to construct all non-split central extensions of ${\bf A}'$. This procedure is:

\begin{enumerate}
\item For a given  Zinbiel  algebra ${\bf A}'$ of dimension $n-s $, determine ${\rm H^{2}}( {\bf A}',\mathbb {C}) $, $\operatorname{Ann}({\bf A}')$ and $\operatorname{Aut}({\bf A}')$.

\item Determine the set of $\operatorname{Aut}({\bf A}')$-orbits on ${\bf T}_{s}({\bf A}') $.

\item For each orbit, construct the  Zinbiel  algebra associated with a
representative of it.
\end{enumerate}

\medskip

\subsubsection{Notations}
Let us introduce the following notations. Let ${\bf A}$ be a nilpotent algebra with
a basis $\{e_{1},e_{2}, \ldots, e_{n}\}.$ Then by $\Delta_{ij}$\ we will denote the
bilinear form
$\Delta_{ij}:{\bf A}\times {\bf A}\longrightarrow \mathbb C$
with $\Delta_{ij}(e_{l},e_{m}) = \delta_{il}\delta_{jm}.$
The set $\left\{ \Delta_{ij}:1\leq i, j\leq n\right\}$ is a basis for the linear space of
bilinear forms on ${\bf A},$ so every $\theta \in
{\rm Z^2}({\bf A},\bf \mathbb V )$ can be uniquely written as $
\theta = \displaystyle \sum_{1\leq i,j\leq n} c_{ij}\Delta _{{i}{j}}$, where $
c_{ij}\in \mathbb C$.
Let us fix the following notations for our nilpotent algebras:

$$\begin{array}{lll}

{\mathfrak N}_{j}& \mbox{---}& j\mbox{th }4\mbox{-dimensional   $2$-step nilpotent algebra.}\\

[\mathfrak{A}]^i_{j}& \mbox{---}& 
j\mbox{th }i\mbox{-dimensional central extension of   $3$-dimensional Zinbiel algebra $\mathfrak{A}$ (see, \cite{centr3zinb}).} 

\end{array}$$
 
\subsection{The algebraic classification of complex $4$-dimensional Zinbiel algebras and their cohomology spaces}

The present table collects all information about multiplication tables and cohomology spaces of $4$-dimensional Zinbiel algebras 
(for the classification of $4$-dimensional Zinbiel algebras see, \cite{kppv} and 
the Corrigendum of  \cite{kppv}), which will be used in our main classification.

\begin{longtable}{ll llllll} 
\hline
\multicolumn{8}{c}{{\bf The list of 2-step nilpotent 4-dimensional Zinbiel algebras}}  \\
\hline
 
{${\mathfrak N}_{01}$} &$:$ &  $e_1e_1 = e_2$ &&&&\\ 
\multicolumn{8}{l}{
${\rm H}^2({\mathfrak N}_{01})=
\Big\langle 
[\Delta_{ 1 3}], [\Delta_{ 1 4}], [\Delta_{ 3 1}], [\Delta_{3 3}], [\Delta_{ 3 4}], [\Delta_{ 4 1}], [\Delta_{4 3}], [\Delta_{ 4 4}], [\Delta_{ 1 2}+2\Delta_{ 2 1}]
\Big\rangle $}\\
 
\hline
{${\mathfrak N}_{02}$} &$:$ & $e_1e_1 = e_3$& $e_2e_2=e_4$  &&&  \\ 

\multicolumn{8}{l}{
${\rm H}^2({\mathfrak N}_{02})=
\Big\langle 
[\Delta_{ 1 2}], [\Delta_{ 2 1}], [\Delta_{ 1 3} + 2\Delta_{ 3 1}], [\Delta_{ 2 4} + 2\Delta_{4 2}]
\Big\rangle $}\\ 

\hline
{${\mathfrak N}_{03}$} &$:$ &  $e_1e_2=  e_3$ & $e_2e_1=-e_3$ &&& \\ 

\multicolumn{8}{l}{
${\rm H}^2({\mathfrak N}_{03})=
\Big\langle 
[\Delta_{ 1 1}], [\Delta_{ 1 2}], [\Delta_{ 1 4}], [\Delta_{ 2 2}], [\Delta_{2 4}], [\Delta_{ 4 1}], [\Delta_{ 4 2}], [\Delta_{ 4 4}]
, [\Delta_{ 1 3}], [\Delta_{ 2 3}], [\Delta_{ 4 3}]
\Big\rangle $}\\

\hline
${\mathfrak N}_{04}^{\alpha}$ &$:$ & $e_1e_1=  e_3$ & $e_1e_2=e_3$& $e_2e_2=\alpha e_3$  &&\\  

\multicolumn{8}{l}{
${\rm H}^2({\mathfrak N}_{04}^{\alpha})=
\Big\langle 
[\Delta_{ 1 2}], [\Delta_{ 1 4}], [\Delta_{ 2 1}], [\Delta_{ 2 2}], [\Delta_{ 2 4}], [\Delta_{ 4 1}], [\Delta_{4 2}], [\Delta_{ 4 4}]
\Big\rangle $}\\

\hline
${\mathfrak N}_{05}$ &$:$ & $e_1e_1=  e_3$& $e_1e_2=e_3$&  $e_2e_1=e_3$ &&\\ 

\multicolumn{8}{l}{
${\rm H}^2({\mathfrak N}_{05})=
\Big\langle 
[\Delta_{ 1 1}], [\Delta_{ 1 2}], [\Delta_{ 1 4}], [\Delta_{2 2}], [\Delta_{ 2 4}], [\Delta_{ 4 1}], [\Delta_{4 2}], [\Delta_{ 4 4}]
\Big\rangle $}\\

\hline
${\mathfrak N}_{06}$ &$:$ & $e_1e_2 = e_4$& $e_3e_1 = e_4$   &&&\\ 

\multicolumn{8}{l}{
${\rm H}^2({\mathfrak N}_{06})=
\Big\langle 
[\Delta_{ 1 1}], [\Delta_{ 1 2}], [\Delta_{ 1 3}], [\Delta_{ 2 1}], [\Delta_{ 2 2}], [\Delta_{ 2 3}], [\Delta_{3 2}], [\Delta_{ 3 3}]
\Big\rangle $}\\

\hline
{${\mathfrak N}_{07}$} &$:$ & $e_1e_2 = e_3$ & $e_2e_1 = e_4$ &  $e_2e_2 = -e_3$ &&\\ 

 \multicolumn{8}{l}{
${\rm H}^2({\mathfrak N}_{07})=
\Big\langle 
[\Delta_{ 1 1}], [\Delta_{ 2 2}], [\Delta_{ 1 3}-\Delta_{ 1 4}-\Delta_{ 2 3}+\Delta_{2 4}-2\Delta_{ 3 2}]
\Big\rangle $}\\


\hline
${\mathfrak N}_{08}^{\alpha}$ &$:$ & $e_1e_1 = e_3$ & $e_1e_2 = e_4$ & $e_2e_1 = -\alpha e_3$ & $e_2e_2 = -e_4$& \\ 
 
\multicolumn{8}{l}{
${\rm H}^2({\mathfrak N}_{08}^{\alpha\neq0,1})=
\Big\langle 
[\Delta_{ 1 2}], [\Delta_{ 2 1}]
\Big\rangle $}\\

\multicolumn{8}{l}{
${\rm H}^2({\mathfrak N}_{08}^{0})=
\Big\langle 
[\Delta_{ 1 2}], [\Delta_{ 2 1}], [\Delta_{ 1 3}+2\Delta_{ 3 1}]
\Big\rangle $}\\

\multicolumn{8}{l}{
{${\rm H}^2({\mathfrak N}_{08}^{1})=
\Big\langle 
[\Delta_{ 1 1}], [\Delta_{ 1 2}],
[\Delta_{ 2 3}-\Delta_{ 1 3}-2\Delta_{ 3 1}+\Delta_{ 3 2}+\Delta_{4 1}], [\Delta_{ 2 4} - \Delta_{ 1 4}-\Delta_{ 3 2} - \Delta_{ 4 1}+2\Delta_{4 2}]
\Big\rangle $}}\\

\hline
${\mathfrak N}_{09}^{\alpha}$ &$:$ & $e_1e_1 = e_4$ & $e_1e_2 = \alpha e_4$ &  $e_2e_1 = -\alpha e_4$ & $e_2e_2 = e_4$ &  $e_3e_3 = e_4$\\

\multicolumn{8}{l}{
${\rm H}^2({\mathfrak N}_{09}^{\alpha})=
\Big\langle 
[\Delta_{ 1 2}], [\Delta_{ 1 3}], [\Delta_{ 2 1}], [\Delta_{ 2 2}], [\Delta_{ 2 3}], [\Delta_{ 3 1}], [\Delta_{ 3 2}], [\Delta_{ 3 3}]
\Big\rangle $}\\

\hline
${\mathfrak N}_{10}$ &$:$ &  $e_1e_2 = e_4$ & $e_1e_3 = e_4$ & $e_2e_1 = -e_4$ & $e_2e_2 = e_4$ & $e_3e_1 = e_4$  \\ 

\multicolumn{8}{l}{
${\rm H}^2({\mathfrak N}_{10})=
\Big\langle 
[\Delta_{ 1 1}], [\Delta_{ 1 2}], [\Delta_{ 1 3}], [\Delta_{ 2 1}], [\Delta_{2 3}], [\Delta_{ 3 1}], [\Delta_{3 2}], [\Delta_{ 3 3}]
\Big\rangle $}\\

\hline
${\mathfrak N}_{11}$ &$:$ &  $e_1e_1 = e_4$ & $e_1e_2 = e_4$ & $e_2e_1 = -e_4$ & $e_3e_3 = e_4$&  \\

\multicolumn{8}{l}{
${\rm H}^2({\mathfrak N}_{11})=
\Big\langle 
[\Delta_{ 1 1}], [\Delta_{ 1 2}], [\Delta_{ 1 3}], [\Delta_{ 2 1}], [\Delta_{2 2}], [\Delta_{2 3}], [\Delta_{ 3 1}], [\Delta_{3 2}]
\Big\rangle $}\\

\hline
{${\mathfrak N}_{12}$} &$:$ &  $e_1e_2 = e_3$ & $e_2e_1 = e_4$  &&& \\ 

 \multicolumn{8}{l}{
${\rm H}^2({\mathfrak N}_{12})=
\Big\langle 
[\Delta_{ 1 1}], [\Delta_{ 2 2}],[\Delta_{ 1 4}-\Delta_{ 1 3}], [\Delta_{ 2 4}-\Delta_{ 2 3}]
\Big\rangle $}\\


\hline
${\mathfrak N}_{13}$ &$:$ & $e_1e_1 = e_4$ & $e_1e_2 = e_3$ & $e_2e_1 = -e_3$ & 
\multicolumn{2}{l}{$e_2e_2=2e_3+e_4$} \\
 
\multicolumn{8}{l}{
${\rm H}^2({\mathfrak N}_{13})=
\Big\langle 
[\Delta_{ 2 1}],[\Delta_{ 2 2}]
\Big\rangle $}\\

\hline
{${\mathfrak N}_{14}^{\alpha}$} &$:$ &   $e_1e_2 = e_4$ & $e_2e_1 =\alpha e_4$ & $e_2e_2 = e_3$&& \\ 

\multicolumn{8}{l}{
${\rm H}^2({\mathfrak N}_{14}^{\alpha\neq-1})=
\Big\langle 
[\Delta_{ 1 1}], [\Delta_{ 2 1}],[\Delta_{ 2 3}+2\Delta_{ 3 2}], [2\alpha\Delta_{ 2 4}+(\alpha+1)\left(\Delta_{1 3}+2\alpha\Delta_{3 1}+2\Delta_{4 2}\right)]
\Big\rangle $}\\ 


\multicolumn{8}{l}{
${\rm H}^2({\mathfrak N}_{14}^{-1})=
\Big\langle 
[\Delta_{ 1 1}], [\Delta_{ 2 1}]
, [\Delta_{ 1 4}], [\Delta_{ 2 4}], [\Delta_{ 2 3}+2\Delta_{ 3 2}]
\Big\rangle $}\\ 


\hline

${\mathfrak N}_{15}$ &$:$ &  $e_1e_2 = e_4$ & $e_2e_1 = -e_4$ & $e_3e_3 = e_4$ && \\

\multicolumn{8}{l}{
${\rm H}^2({\mathfrak N}_{15})=
\Big\langle 
[\Delta_{ 1 1}], [\Delta_{ 1 3}], [\Delta_{ 2 1}], [\Delta_{ 2 2}], [\Delta_{2 3}], [\Delta_{ 3 1}], [\Delta_{ 3 2}], [\Delta_{3 3}]
\Big\rangle $}\\
\hline
\multicolumn{8}{c}{
\bf The list of 3-step nilpotent 4-dimensional Zinbiel algebras} \\
\hline

{$\mathfrak{Z}_1$} &$:$& $ e_1e_1=e_2$ & $e_1 e_2 =e_3$ &$e_2 e_1 =2e_3$ &\\ 

\multicolumn{8}{l}{
${\rm H}^2(\mathfrak{Z}_1)=
\Big\langle 
[\Delta_{13}+3 \Delta_{22}+3\Delta_{31}], [\Delta_{ 1 4}], [\Delta_{ 41}], [\Delta_{ 4 4}], 
\Big\rangle $}\\

\hline

$[\mathfrak{N}_1^{\mathbb{C}}]^1_{01}$ &$:$& $ e_1e_1=e_2$ & $e_1 e_2 =e_4$ &$e_2 e_1 =2e_4$ & $e_3 e_3 =e_4$ &\\ 

\multicolumn{8}{l}{
${\rm H}^2([\mathfrak{N}_1^{\mathbb{C}}]^1_{01})=
\Big\langle 
[\Delta_{ 1 3}], [\Delta_{ 3 1}]
\Big\rangle $}\\

\hline

$[\mathfrak{N}_1^{\mathbb{C}}]^1_{02}$ &$:$& $ e_1e_1=e_2$ & $e_1 e_2 =e_4$ &$e_1 e_3 =e_4$ &$e_2 e_1 =2e_4$ &\\ 

\multicolumn{8}{l}{
${\rm H}^2([\mathfrak{N}_1^{\mathbb{C}}]^1_{02})=
\Big\langle 
[\Delta_{ 1 3}], [\Delta_{ 3 1}], [\Delta_{ 3 3}]
\Big\rangle $}\\

\hline

$[\mathfrak{N}_1]^1_{01}$ &$:$& $e_1e_2=  e_3$&$e_1e_3=e_4$
&$ e_2e_1= -e_3$ &$e_2e_2=e_4$ &\\

\multicolumn{8}{l}{
${\rm H}^2([\mathfrak{N}_1]^1_{01})=
\Big\langle 
[\Delta_{ 1 1}], [\Delta_{ 1 2}], [\Delta_{ 1 3}], [\Delta_{ 2 3}]
\Big\rangle $}\\

\hline

$[\mathfrak{N}_1]^1_{02}$ &$:$& $e_1e_2=  e_3$&$e_1e_3=e_4$
&$ e_2e_1= -e_3$ &&\\ 

\multicolumn{8}{l}{
${\rm H}^2([\mathfrak{N}_1]^1_{02})=
\Big\langle 
[\Delta_{ 1 1}], [\Delta_{ 1 2}], [\Delta_{ 2 2}], [\Delta_{ 2 3}]
\Big\rangle $}\\

\hline

\multicolumn{8}{c}{\bf The list of 4-dimensional null-filiform Zinbiel algebras}   \\
\hline

{$[\mathfrak{Z}_1]^1_1$} &$:$& 
$e_1e_1=e_2$ & $e_1e_2=\textstyle\frac{1}{2}e_3$ & $e_1e_3=2e_4$   & $e_2e_1=e_3$ & $e_2e_2=3e_4$ & $e_3e_1=6e_4$ \\

\multicolumn{8}{l}{
${\rm H}^2([\mathfrak{Z}_1]^1_1)=
\Big\langle 
[\Delta_{ 1 4}+8\Delta_{ 2 3}+12\Delta_{ 3 2}+4\Delta_{ 4 1}]
\Big\rangle $}\\

\hline
\end{longtable}

 \begin{remark}
From the previous list, we will only consider ${\mathfrak N}_{01}$, ${\mathfrak N}_{02}$, ${\mathfrak N}_{03}$, ${\mathfrak N}_{07}$, ${\mathfrak N}_{08}^{1}$, ${\mathfrak N}_{12}$, ${\mathfrak N}_{14}^{\alpha}$, $\mathfrak{Z}_1$, $[\mathfrak{Z}_1]^1_1$. All $n$-dimensional central extensions of the remaining algebras are split or have $(n + 1)$-dimensional annihilator.
\end{remark} 

\begin{remark}
In what follows, when we consider an automorphism, the entries of the matrix that are choosen to be zero will be omitted in the description.
\end{remark}

\subsubsection{Central extensions of ${\mathfrak N}_{01}$}

Let us use the following notations:
\begin{longtable}{lllll}
$\nabla_1 =[\Delta_{12}+2\Delta_{21}],$& $\nabla_2 = [\Delta_{13}],$ & $ \nabla_3 = [\Delta_{14}],$ &
$\nabla_4 = [\Delta_{31}],$ & $\nabla_5 = [\Delta_{33}],$\\
$\nabla_6 = [\Delta_{34}],$ & $ \nabla_7 = [\Delta_{41}],$ &
$\nabla_8 = [\Delta_{43}],$ & $\nabla_9 = [\Delta_{44}].$
\end{longtable}
Take $\theta=\sum\limits_{i=1}^9\alpha_i\nabla_i\in {\rm H}^2({\mathfrak N}_{01}).$ 
	The automorphism group of ${\mathfrak N}_{01}$ is generated by invertible matrices of the form

\begin{center}$	\phi=
	\begin{pmatrix}
	x &  0 &  0 & 0\\
	y &  x^2  &  z & w\\
	q &  0  &  s & t \\
	r &  0  &  u & v \\
	\end{pmatrix}.$ \end{center}

	Since
	$$
	\phi^T\begin{pmatrix}
    0 &  \alpha_1 & \alpha_2 & \alpha_3\\
	2\alpha_1  &  0 & 0 & 0 \\
	\alpha_4&  0  & \alpha_5 & \alpha_6\\
	\alpha_7 & 0 & \alpha_8 & \alpha_9
	\end{pmatrix} \phi=	\begin{pmatrix}
\alpha^\ast &  \alpha_1^\ast & \alpha_2^\ast & \alpha_3^\ast\\
	2\alpha_1^\ast  &  0 & 0 & 0 \\
	\alpha_4^\ast &  0  & \alpha_5^\ast & \alpha_6^\ast\\
	\alpha_7^\ast & 0 & \alpha_8^\ast & \alpha_9^\ast
	\end{pmatrix},
	$$	
we have that the action of ${\rm Aut} ({\mathfrak N}_{01})$ on the subspace
$\langle \sum\limits_{i=1}^9\alpha_i\nabla_i  \rangle$
is given by
$\langle \sum\limits_{i=1}^9\alpha_i^{*}\nabla_i\rangle,$
where

\begin{longtable}{lcl}
$\alpha_1^\ast$&$=$&$   \alpha_1x^3,$ \\
$\alpha_2^\ast$&$=$&$ \alpha_1 x z+\alpha_2 s x + \alpha_3 u x + \alpha_5 q s+ \alpha_6 q u + \alpha_8 r s  + \alpha_9ru,$ \\
$\alpha_3^\ast$&$=$&$ \alpha_1 wx+ \alpha_2 tx + \alpha_3 vx + \alpha_5 q t +  \alpha_6 q v +\alpha_8 r t +\alpha_9 r v ,$ \\
$\alpha_4^\ast$&$=$&$  2\alpha_1x z+\alpha_4 s x +\alpha_5 q s + \alpha_6 r s + \alpha_7 u x +\alpha_8 qu + \alpha_9 ru,$\\
$\alpha_5^\ast$&$=$&$ \alpha_5 s^2 +  (\alpha_6 s + \alpha_8 s + \alpha_9 u)u,$\\
$\alpha_6^\ast$&$=$&$ \alpha_5 s t + \alpha_6 s v +\alpha_8 t u +  \alpha_9 u v,$\\
$\alpha_7^\ast$&$=$&$ 2 \alpha_1 w x+\alpha_4 tx +\alpha_5 qt + \alpha_6 rt + \alpha_7 vx +\alpha_8 qv + \alpha_9 r v,$\\
$\alpha_8^\ast$&$=$&$ \alpha_5 s t + \alpha_6 t u + \alpha_8 s v + \alpha_9 u v,$\\
$\alpha_9^\ast$&$=$&$ \alpha_5 t^2 +  (\alpha_6 t + \alpha_8 t + \alpha_9 v)v.$
\end{longtable}

First note that we can not have $\alpha_1=0$, $\alpha_2=\alpha_4=\alpha_5=\alpha_6=\alpha_8=0$ or $\alpha_3=\alpha_6=\alpha_7=\alpha_8=\alpha_9=0$. Besides, considering $\phi$ given by 
 $x=s=v=1$, $z=-\frac{\alpha_4}{2\alpha_1}$ \ and \ $w=-\frac{\alpha_7}{2\alpha_1}$, 
we can suppose $\alpha_4=\alpha_7=0$. On the other hand, if we choose $\phi$ given by $x=u=t=1$ and $z=w$ we have the following change of position: $\alpha_2\leftrightarrow \alpha_3$, $\alpha_5\leftrightarrow \alpha_9$ and $\alpha_6\leftrightarrow \alpha_8$. Finally, we note that if $\alpha_5\neq 0$ then we can choose $\alpha_8=0$ (just take $\phi$ given by $x=s=v=1$ and $t=-\frac{\alpha_8}{\alpha_5}$) and applying the automorphism given by $x=s=v=1$, $z=\frac{\alpha_6 \alpha_3}{\alpha_9}$, $w=\alpha_3$ and $r=-\frac{2\alpha_3}{\alpha_9}$, we can suppose $\alpha_3=0$. In what follow, we will consider $\alpha_1=1$ and $\alpha_4=\alpha_7=0$. Then:

     \begin{enumerate}

             \item Suppose $\alpha_5=\alpha_9=0$ and $\alpha_6=-\alpha_8$.
             \begin{enumerate}
                 \item If $\alpha_6=0$, note that $\alpha_3\neq 0$. Therefore, by choosing $x=s=1$, $u=-\frac{\alpha_2}{\alpha_3}$ and $v=\frac{1}{\alpha_3}$, we obtain the representative $\langle \nabla_1 +\nabla_3 \rangle$, that we do not consider since the associated algebra would have $2$-dimensional annihilator.
                 \item If $\alpha_6\neq 0$, by choosing $x=v=-u=1$, $z=\frac{\alpha_3}{3}$, $w=-\frac{\alpha_2+\alpha_3\alpha_6}{3 \alpha_6}$, $q=-\frac{2\alpha_3}{3 \alpha_6}$, $t=\frac{1}{\alpha_6}$ and $r=\frac{2\alpha_2}{3 \alpha_6}$, we obtain the representative $\langle \nabla_1+\nabla_6-\nabla_8  \rangle$.
             \end{enumerate}
             \item Suppose $\alpha_5=\alpha_9=0$ and $\alpha_6\neq -\alpha_8$.
             \begin{enumerate}
                 \item If $\alpha_6=2\alpha_8$, we have $\alpha_8\neq 0$. Then, there are two cases to consider:
                 \begin{enumerate}
                     \item If $\alpha_2=0$, by choosing $x=t=1$, $u=\frac{1}{\alpha_8}$, $z=\frac{\alpha_3}{3\alpha_8}$ and $q=-\frac{2\alpha_3}{3\alpha_8}$, we obtain the representative $\langle \nabla_1+\nabla_6+2\nabla_8  \rangle$.
                     \item If $\alpha_2\neq 0$, by choosing $x=1$, $t=\frac{1}{\alpha_2}$, $u=\frac{\alpha_2}{\alpha_8}$, $z=\frac{\alpha_2\alpha_3}{3\alpha_8}$ and $r=-\frac{2\alpha_3}{3\alpha_8}$, we obtain the representative $\langle \nabla_1+\nabla_3+\nabla_6+2\nabla_8  \rangle$.
                 \end{enumerate}
                 \item Suppose $\alpha_6\neq 2\alpha_8$.
                  \begin{enumerate}
                     \item Suppose $\alpha_8= 2\alpha_6$ and, thus, $\alpha_6\neq 0$.
                     \begin{enumerate}
                         \item If $\alpha_3=0$, by choosing $x=v=1$, $s=\frac{1}{\alpha_6}$, $z=\frac{\alpha_2}{3\alpha_6}$ and $r=-\frac{2\alpha_2}{3\alpha_6}$, we obtain the representative $\langle \nabla_1+\nabla_6+2\nabla_8  \rangle$.
                         \item If $\alpha_3\neq 0$, by choosing $x=1$, $s=\frac{\alpha_3}{\alpha_6}$, $v=\frac{1}{\alpha_3}$, $z=\frac{\alpha_2\alpha_3}{3\alpha_6}$ and $r=-\frac{2\alpha_2}{3\alpha_6}$, we obtain the representative $\langle \nabla_1+\nabla_3+\nabla_6+2\nabla_8  \rangle$.
                     \end{enumerate}
                     \item If $\alpha_8\neq 2\alpha_6$, since we can change position between $\alpha_6$ and $\alpha_8$, we can suppose $\alpha_8\neq 0$. Then, by choosing $x=u=1$, $t=\frac{1}{\alpha_8}$, $z=\frac{\alpha_3\alpha_8}{2\alpha_6 - \alpha_8}$ $w=\frac{\alpha_2\alpha_6}{(2\alpha_8-\alpha_6)\alpha_8}$, $q=\frac{2\alpha_3}{\alpha_8-2\alpha_6}$ and  $r=\frac{2\alpha_2}{\alpha_6 - 2 \alpha_8}$, we obtain the family of representatives $\langle \nabla_1+\nabla_6+\alpha\nabla_8  \rangle$, with $\alpha\neq -1, \frac{1}{2}, 2$. Note that for $\alpha \neq 0$, $\langle \nabla_1+\nabla_6+\alpha\nabla_8  \rangle$ and $\langle \nabla_1+\nabla_6+\frac{1}{\alpha}\nabla_8  \rangle$ are in the same orbit.
                 \end{enumerate}
             \end{enumerate} 
        \item Suppose $\alpha_5= 0$ and $\alpha_9\neq 0$.
         \begin{enumerate}
             \item Suppose $\alpha_8= -\alpha_6$.
             \begin{enumerate}
                 \item If $\alpha_6\neq 0$, by choosing $x=\sqrt[3]{\alpha_6^2\alpha_9}$, $z=-\frac{3 \alpha_3 \alpha_6 + 4 \alpha_2 \alpha_9}{9}$, $w=-\frac{\alpha_2 \alpha_9}{3}$, $q=-\frac{2 \sqrt[3]{\alpha_9} (3 \alpha_3 \alpha_6 + \alpha_2 \alpha_9)}{9 \sqrt[3]{\alpha_6^4}}$, $t=\alpha_9$, $r=\frac{2 \alpha_2 \sqrt[3]{\alpha_9}}{3 \sqrt[3]{\alpha_6}}$ and $u=-\alpha_6$, we obtain the representative $\langle \nabla_1+\nabla_5+\nabla_6-\nabla_8  \rangle$.
                 \item If $\alpha_6=0$, note that $\alpha_2\neq 0$. By choosing $x=1$, $s=\frac{1}{\alpha_2}$, $t=-\frac{\alpha_3}{\alpha_2\sqrt{\alpha_9}}$  and $v=\frac{1}{\sqrt{\alpha_9}}$, we obtain the representative $\langle \nabla_1+\nabla_2+\nabla_9  \rangle$.
             \end{enumerate}
             \item If $\alpha_8\neq -\alpha_6$, by choosing $x=s=t=1$ and $v=-\frac{\alpha_6+\alpha_8}{\alpha_9}$, we return to the case $\alpha_5=\alpha_9=0$.
         \end{enumerate}     
        \item  Suppose $\alpha_5\alpha_9\neq 0$. By choosing $x=t=u=1$ and $s=-\frac{\alpha_6 + \alpha_8 + \sqrt{(\alpha_6 + \alpha_8)^2 - 4 \alpha_5 \alpha_9}}{2 \alpha_5}$, we return to the case $\alpha_5=0$ and $\alpha_9\neq 0$
    \end{enumerate}

Summarizing, we have the following distinct orbits
\begin{center}
    $\langle \nabla_1+\nabla_2+\nabla_9  \rangle,$
    $\langle \nabla_1+\nabla_6+\alpha\nabla_8  \rangle^{O(\alpha)=O(\frac{1}{\alpha})}$,
    $\langle \nabla_1+\nabla_5+\nabla_6-\nabla_8  \rangle,$
    $\langle \nabla_1+\nabla_3+\nabla_6+2\nabla_8  \rangle$,
\end{center}
which give the following new algebras:

\begin{longtable}{llllllll}
$\mathcal{Z}_{01}$ &$:$& $ e_1e_1=e_2$ & $e_1 e_2 =e_5$ & $e_1 e_3 =e_5$ & $e_2e_1=2e_5$ & $e_4e_4 =e_5$ \\
\hline
$\mathcal{Z}_{02}^\alpha$ &$:$& $ e_1e_1=e_2$ & $e_1 e_2 =e_5$ & $e_2e_1=2e_5$ & $e_3e_4=e_5$ & $e_4e_3 =\alpha e_5$ \\
\hline
$\mathcal{Z}_{03}$ &$:$& $ e_1e_1=e_2$ & $e_1 e_2 =e_5$ & $e_2e_1=2e_5$  & $e_3e_3=e_5$ & $e_3e_4=e_5$ & $e_4e_3 =-e_5$ \\
\hline
$\mathcal{Z}_{04}$ &$:$& $ e_1e_1=e_2$ & $e_1 e_2 =e_5$ & $e_1e_4=e_5$  & $e_2e_1=2e_5$ & $e_3e_4=e_5$ & $e_4e_3=2e_5$ 
\end{longtable}

\subsubsection{Central extensions of ${\mathfrak N}_{02}$}

Let us use the following notations:
\begin{longtable}{llll}
$\nabla_1 =[\Delta_{12}],$& $\nabla_2 = [\Delta_{21}],$ & $ \nabla_3 = [\Delta_{ 13} + 2\Delta_{31}],$ &
$\nabla_4 = [\Delta_{ 2 4} + 2\Delta_{4 2}].$
\end{longtable}
Take $\theta=\sum\limits_{i=1}^4\alpha_i\nabla_i\in {\rm H}^2({\mathfrak N}_{02}).$ 
	The automorphism group of ${\mathfrak N}_{02}$ is generated by invertible matrices of the form

\begin{center} $\phi_1=
	\begin{pmatrix}
	x &    0  &  0 & 0\\
	0 &  y  &  0 & 0\\
	z &  t  &  x^2 & 0 \\
	u &  v  &  0 & y^2 \\
	\end{pmatrix}$ and $	\phi_2=
	\begin{pmatrix}
	0 &  x &  0 & 0\\
	y &  0  &  0 & 0\\
	z &  t  &  0 & x^2 \\
	u &  v  &  y^2 & 0 \\
	\end{pmatrix}.$\end{center}

	Since
	$$
	\phi^T_1\begin{pmatrix}
    0 &  \alpha_1 & \alpha_3 & 0\\
	\alpha_2  &  0 & 0 & \alpha_4 \\
	2\alpha_3&  0  & 0 & 0\\
	0 & 2\alpha_4 & 0 & 0
	\end{pmatrix} \phi_1=	\begin{pmatrix}
\alpha^\ast &  \alpha_1^\ast & \alpha_3^\ast & 0\\
	\alpha_2^\ast  &  \alpha^{\ast\ast} & 0 & \alpha_4^\ast \\
	2\alpha_3^\ast&  0  & 0 & 0\\
	0 & 2\alpha_4^\ast & 0 & 0
	\end{pmatrix},
	$$	
we have that the action of ${\rm Aut} ({\mathfrak N}_{02})$ on the subspace
$\langle \sum\limits_{i=1}^4\alpha_i\nabla_i  \rangle$
is given by
$\langle \sum\limits_{i=1}^4\alpha_i^{*}\nabla_i\rangle,$
where 
\begin{longtable}{llll}
$\alpha_1^\ast=  \alpha_1 x y+ \alpha_3 t x + 2 \alpha_4 u y, $ &
$\alpha_2^\ast=  \alpha_2 x y + 2 \alpha_3 t x + \alpha_4 u y, $ &
$\alpha_3^\ast=  \alpha_3x^3,$ &
$\alpha_4^\ast=  \alpha_4y^3.$
\end{longtable}

Since $\operatorname{Ann}({\mathfrak N}_{02})=\langle e_3,e_4 \rangle$, we can not have $\alpha_3\alpha_4=0$. Thus, we can suppose $\alpha_4=1$. Then, by choosing $x=\frac{1}{\sqrt[3]{\alpha_3}}$, $y=1$, $t=\frac{\alpha_1 - 2 \alpha_2}{3 \alpha_3}$ and 
$u=\frac{ \alpha_2-2 \alpha_1 }{3 \sqrt[3]{\alpha_3}}$, we obtain the representative  
\begin{center}
$\langle \nabla_3+\nabla_4  \rangle$,
\end{center}
which gives the following new algebra:

\begin{longtable}{llllllll}
$\mathcal{Z}_{05}$ &$:$& $ e_1e_1=e_3$ & $e_1 e_3 =e_5$ & $e_2e_2 =e_4$ & $e_2 e_4 =e_5$ & $e_3e_1=2e_5$ & $e_4e_2=2e_5$
\end{longtable}


\subsubsection{Central extensions of ${\mathfrak N}_{03}$}

Let us use the following notations:
\begin{longtable}{llllll}
$\nabla_1 =[\Delta_{11}],$& $\nabla_2 = [\Delta_{12}],$ & $\nabla_3 = [\Delta_{13}],$  & $ \nabla_4 = [\Delta_{14}],$ &
$\nabla_5 = [\Delta_{22}],$ \\
$\nabla_6 = [\Delta_{23}],$ & $\nabla_7 = [\Delta_{24}],$ &
$\nabla_8 = [\Delta_{41}],$ & $ \nabla_9 = [\Delta_{42}],$ &
 $\nabla_{10} = [\Delta_{43}],$ & $\nabla_{11} = [\Delta_{44}].$
\end{longtable}
Take $\theta=\sum\limits_{i=1}^{11}\alpha_i\nabla_i\in {\rm H}^2({\mathfrak N}_{03}).$ 
	The automorphism group of ${\mathfrak N}_{03}$ is generated by invertible matrices of the form

\begin{center} 
$	\phi=
	\begin{pmatrix}
	x &   w  &  0 & 0\\
	z &  y  &  0 & 0\\
	q &  r  &  xy-wz & s \\
	t &  u  &  0 & v \\
	\end{pmatrix}.$\end{center}

	Since
	$$
	\phi^T\begin{pmatrix}
    \alpha_1 &  \alpha_2 & \alpha_3 & \alpha_4\\
	0  &  \alpha_5 & \alpha_6 & \alpha_7 \\
	0 &  0  & 0 & 0\\
	\alpha_8 & \alpha_9 & \alpha_{10} & \alpha_{11}
	\end{pmatrix} \phi=	\begin{pmatrix}
\alpha_1^\ast &  \alpha_2^\ast+\alpha^\ast & \alpha_3^\ast & \alpha_4^\ast\\
	-\alpha^\ast  &  \alpha_5^\ast & \alpha_6^\ast & \alpha_7^\ast \\
	0 &  0  & 0 & 0\\
	\alpha_8^\ast & \alpha_9^\ast & \alpha_{10}^\ast & \alpha_{11}^\ast
	\end{pmatrix},
	$$	
we have that the action of ${\rm Aut} ({\mathfrak N}_{03})$ on the subspace
$\langle \sum\limits_{i=1}^{11}\alpha_i\nabla_i  \rangle$
is given by
$\langle \sum\limits_{i=1}^{11}\alpha_i^{*}\nabla_i\rangle,$
where 

\begin{longtable}{lll}
$\alpha_1^\ast $ & $=$ & $   \alpha_1 x^2 + \alpha_2 xz + \alpha_3 xq + 
 \alpha_4 xt +\alpha_5 z^2+\alpha_6 zq + \alpha_7 zt + \alpha_8 xt + \alpha_9 zt  + \alpha_{10} qt  + \alpha_{11} t^2,   $ \\
$\alpha_2^\ast$ & $=$ & $  2\alpha_1wx + \alpha_2 (x y+wz) +  \alpha_3(qw+ r x) + \alpha_4 (tw+u x) + 2\alpha_5 z y + \alpha_6 (r z   + q y)+ $\\
$ $ & $ $ & \multicolumn{1}{r}{$
 \alpha_7 (t y+ z u) + \alpha_8 (tw+u x) +\alpha_9 (t y + z u) +\alpha_{10} (r t + q u)+ 2\alpha_{11} t u,$} \\
$\alpha_3^\ast$ & $=$ & $   (\alpha_3 x+\alpha_6 z + \alpha_{10} t ) (xy-wz),$ \\
$\alpha_4^\ast$ & $=$ & $   \alpha_3 s x + \alpha_4 x v+\alpha_6 z s + \alpha_7 z v + \alpha_{10} s t +  \alpha_{11} t v ,$\\
$\alpha_5^\ast $ & $=$ & $    (\alpha_3r+\alpha_4u+\alpha_8u)w+(\alpha_5 y+\alpha_6 r + \alpha_7 u + \alpha_9 u )y+ \alpha_{10} r u + \alpha_{11} u^2, $\\
$\alpha_6^\ast $ & $=$ & $    (\alpha_3 w+\alpha_6 y+\alpha_{10} u )(xy-wz) , $\\
$\alpha_7^\ast $ & $=$ & $    \alpha_3 sw+\alpha_4 vw+ \alpha_6 s y+ \alpha_7 y v+\alpha_{10} s u  + \alpha_{11} u v, $\\
$\alpha_8^\ast $ & $=$ & $    (\alpha_8 x+\alpha_9 z + \alpha_{10} q + \alpha_{11}t ) v,$\\
$\alpha_9^\ast $ & $=$ & $    (\alpha_8 w+\alpha_9 y+\alpha_{10} r+\alpha_{11} u) v,$\\
$\alpha_{10}^\ast$ & $=$ & $   \alpha_{10}(xy-wz)v,$\\
$\alpha_{11}^\ast$ & $=$ & $   (\alpha_{10}s+\alpha_{11} v)v.$
\end{longtable}

First note that we can not have $\alpha_3=\alpha_6=\alpha_{10}=0$ or $\alpha_4=\alpha_7=\alpha_8=\alpha_9=\alpha_{10}=\alpha_{11}=0$. Besides, a standard computation shows that we can change the position of $\alpha_3$ and $\alpha_6$.


Since $\alpha_{10}^*=0$ if, and only if, $\alpha_{10}=0$, we have:

    \begin{enumerate}
        \item Suppose $\alpha_{10}=\alpha_{11}=0$. Note that in this case we can not have $\alpha_3=\alpha_6=0$.

        \begin{enumerate}

            \item 
            
             

            Suppose $\alpha_3\alpha_6\neq 0$ and $\alpha_9=0$. Write $H=\alpha_3^2 \alpha_5- \alpha_2 \alpha_3 \alpha_6 + \alpha_1 \alpha_6^2$.
            \begin{enumerate}
                    \item If $\alpha_8=0$ and $\alpha_4 \alpha_6 \neq \alpha_3 \alpha_7$, by choosing $x=1$, $w=-\frac{\alpha_6}{\alpha_3^2}$, $y=\frac{1}{\alpha_3}$, $q=-\frac{\alpha_1}{\alpha_3}$, $r=\frac{ 2 \alpha_1 \alpha_3 \alpha_6 \alpha_7-\alpha_1 \alpha_4 \alpha_6^2 + 
 \alpha_3^2 (\alpha_4 \alpha_5 - \alpha_2 \alpha_7)}{\alpha_3^3 (\alpha_3 \alpha_7-\alpha_4 \alpha_6)}$, $s=\frac{\alpha_3  \alpha_4}{\alpha_4 \alpha_6-\alpha_3 \alpha_7}$, $u=\frac{H}{\alpha_3^2 (\alpha_3 \alpha_7-\alpha_4 \alpha_6)}$ and $v=\frac{\alpha_3^2}{\alpha_3 \alpha_7-\alpha_4 \alpha_6}$, we obtain the representative $\langle \nabla_3+\nabla_7  \rangle$.
                    \item Suppose $\alpha_8=0$ and $\alpha_4 \alpha_6 =\alpha_3 \alpha_7$.
                    \begin{enumerate}
                        \item If $H=0$, by choosing $x=v=1$, $w=-\frac{\alpha_6}{\alpha_3^2}$, $y=\frac{1}{\alpha_3}$, $q=-\frac{\alpha_1}{\alpha_3}$, $r=\frac{ 2 \alpha_1  \alpha_6 - 
 \alpha_2 \alpha_3}{\alpha_3^3}$ and $s=-\frac{\alpha_7}{\alpha_6}$, we obtain the representative $\langle \nabla_3  \rangle$, that we do not consider since it does not satisfy the above condition.
                        \item If  $H\neq 0$, by choosing $x=\frac{\sqrt[4]{H}}{\alpha_3}$, $w=-\frac{\alpha_6}{\sqrt{H}}$, $y=\frac{\alpha_3}{\sqrt{H}}$, $q=-\frac{\alpha_1\sqrt[4]{H}}{\alpha_3^2}$, $r=\frac{ 2 \alpha_1  \alpha_6 - 
 \alpha_2 \alpha_3}{\alpha_3\sqrt{H}}$, $s=-\frac{\alpha_7}{\alpha_6}$ and $v=1$, we obtain the representative $\langle \nabla_3+\nabla_5  \rangle$, that we do not consider since it does not satisfy the above condition.
                    \end{enumerate}
                    \item Suppose $\alpha_8\neq 0$ and $\alpha_4 \alpha_6 + \alpha_6 \alpha_8\neq \alpha_3 \alpha_7$. By choosing $x=\frac{1}{\alpha_6}$, $z=-\frac{\alpha_3}{\alpha_6^2}$, $y=1$, $q=\frac{\alpha_1 \alpha_6^2 \alpha_7-\alpha_3^2 \alpha_5\alpha_7 + 2 \alpha_3 \alpha_5 \alpha_6 (\alpha_4 + \alpha_8) - \alpha_2 \alpha_6^2 (\alpha_4 + \alpha_8)}{\alpha_6^3(\alpha_4 \alpha_6 - \alpha_3 \alpha_7 + \alpha_6 \alpha_8)}$, $r=-\frac{\alpha_5}{\alpha_6}$, $s=-\frac{\alpha_7}{\alpha_8}$, $t=\frac{\alpha_2 \alpha_3 \alpha_6-\alpha_3^2 \alpha_5 - \alpha_1 \alpha_6^2}{\alpha_6^2 (\alpha_4 \alpha_6 - \alpha_3 \alpha_7 + \alpha_6 \alpha_8)}$ and $v=\frac{\alpha_6}{\alpha_8}$, we obtain the family of representatives $\langle \alpha\nabla_4+\nabla_6+\nabla_8  \rangle$.
                    \item Suppose $\alpha_8\neq 0$ and $\alpha_4 \alpha_6 + \alpha_6 \alpha_8= \alpha_3 \alpha_7$.
                    \begin{enumerate}
                        \item If $H\neq 0$, by choosing $x=\frac{\alpha_6}{\sqrt{H}}$, $z=-\frac{\alpha_3}{\sqrt{H}}$, $y=\frac{\sqrt[4]{H}}{\alpha_6}$, $q=\frac{2\alpha_3\alpha_5-\alpha_2 \alpha_6}{\alpha_6\sqrt{H}}$, $r=-\frac{\alpha_5\sqrt[4]{H}}{\alpha_6^2}$, $s=-\frac{\alpha_7\sqrt{H}}{\alpha_6^2\alpha_8}$ and $v=\frac{\sqrt{H}}{\alpha_6\alpha_8}$, we obtain the representative $\langle \nabla_1-\nabla_4+\nabla_6+\nabla_8  \rangle$.
                        \item If $H=0$, by choosing $x=1$, $z=-\frac{\alpha_3}{\alpha_6}$, $y=\frac{1}{\sqrt{\alpha_6}}$, $q=\frac{2\alpha_3\alpha_5-\alpha_2 \alpha_6}{\alpha_6^2}$, $r=-\frac{\alpha_5}{\sqrt{\alpha_6^3}}$, $s=-\frac{\alpha_7}{\alpha_6\alpha_8}$ and $v=\frac{1}{\alpha_8}$, we obtain the representative $\langle -\nabla_4+\nabla_6+\nabla_8  \rangle$.
                    \end{enumerate}
            \end{enumerate}
            
            \item Suppose $\alpha_6=\alpha_9= 0$ and $\alpha_3\neq 0$.
            \begin{enumerate}
                \item If $\alpha_7\alpha_8\neq 0$, by choosing    $x=\alpha_7$, $y=\alpha_8$, $q=-\frac{\alpha_1\alpha_7}{\alpha_3}$, $r=\frac{\alpha_8 (\alpha_4 \alpha_5 - \alpha_2 \alpha_7 + \alpha_5 \alpha_8)}{\alpha_3\alpha_7}$, $s=-\alpha_4\alpha_7$, $q=-\frac{\alpha_5\alpha_8}{\alpha_7}$ and $v=\alpha_3\alpha_7$, we obtain the representative $\langle \nabla_3+\nabla_7+\nabla_8  \rangle$.
                \item If $\alpha_5\alpha_8\neq 0$ and $\alpha_7=0$, by choosing $x=\alpha_5$, $y=\alpha_3\alpha_5$, $q=-\frac{\alpha_1\alpha_5}{\alpha_3}$, $r=-\alpha_2\alpha_5$, $s=-\frac{\alpha_3\alpha_4\alpha_5^2}{\alpha_8}$ and  $v=\frac{\alpha_3^2\alpha_5^2}{\alpha_8}$     , we obtain the representative $\langle \nabla_3+\nabla_5+\nabla_8  \rangle$.
                \item If $\alpha_8\neq 0$ and $\alpha_5=\alpha_7=0$, by choosing $x=1$, $y=\alpha_8$, $q=-\frac{\alpha_1}{\alpha_3}$, $r=-\frac{\alpha_2\alpha_8}{\alpha_3}$, $s=-\alpha_4$ and  $v=\alpha_3$, we obtain the representative $\langle \nabla_3+\nabla_8  \rangle$.
                \item If $\alpha_8= 0$ and $\alpha_7\neq 0$, by choosing $x=\alpha_7$, $y=1$, $q=-\frac{\alpha_1\alpha_7}{\alpha_3}$, $r=\frac{\alpha_4\alpha_5-\alpha_2\alpha_7}{\alpha_3\alpha_7}$, $s=-\alpha_4\alpha_7$, $u=-\frac{\alpha_5}{\alpha_7}$ and  $v=\alpha_3\alpha_7$, we obtain the representative $\langle \nabla_3+\nabla_7  \rangle$.
                \item If $\alpha_7\alpha_8= 0$, by choosing   $v=1$ and $s=-\frac{\alpha_4}{\alpha_3}$, we obtain the representative 
                $\langle  
                \alpha_1\nabla_1+ \alpha_2\nabla_2+ \alpha_3\nabla_3+ \alpha_5\nabla_5  \rangle$, that we do not consider since it does not satisfy the above condition.
            \end{enumerate}
            
            \item If $\alpha_3\alpha_9\neq 0$ or $\alpha_6\alpha_9\neq 0$, by choosing $w=z=v=1$ and $y=-\frac{\alpha_8}{\alpha_9}$, we return to one of the above cases.

             \item If $\alpha_3=\alpha_9= 0$ and $\alpha_6\neq 0$, by choosing $x=y=z=v=1$, we return to the case when $\alpha_3\alpha_6\neq 0$ and $\alpha_9=0$.

        \end{enumerate}
        

        \item Suppose $\alpha_{10}=0$ and $\alpha_{11}\neq 0$. Again, in this case we can not have $\alpha_3=\alpha_6=0$.
            \begin{enumerate}
             \item Suppose $\alpha_6\neq 0$ and $\alpha_3= 0$.
                \begin{enumerate}
                    \item If $\alpha_8\neq \alpha_4$, by choosing $x=\frac{\sqrt{\alpha_{11}}}{\alpha_4-\alpha_8}$, $y=\frac{\sqrt{\alpha_4-\alpha_8}}{\sqrt[4]{\alpha_{11}\alpha_{6}^2}}$, $q=\frac{\alpha_7 \alpha_8 -\alpha_{11} \alpha_2 + \alpha_4 \alpha_9}{\alpha_{6}\sqrt{\alpha_{11}}(\alpha_4-\alpha_8)}$, $r=\frac{\sqrt{\alpha_4-\alpha_8}(\alpha_7 \alpha_9-\alpha_{11} \alpha_5)}{\sqrt[4]{\alpha_{11}^5\alpha_6^6}}$, $s=\frac{  \alpha_9-\alpha_7}{\alpha_{6}\sqrt{\alpha_{11}}}$, $t=\frac{\alpha_8}{\sqrt{\alpha_{11}}(\alpha_8-\alpha_4)}$, $u=\frac{\alpha_9\sqrt{\alpha_4-\alpha_8}}{\sqrt[4]{\alpha_6^2\alpha_{11}^5}}$ and $v=\frac{1}{\sqrt{\alpha_{11}}}$, we obtain the family of representatives $\langle \alpha\nabla_1+\nabla_4+\nabla_6+\nabla_{11}  \rangle$.
                    \item If $\alpha_8= \alpha_4$ and $\alpha_1\alpha_{11}\neq\alpha_4^2$, by choosing $x=\frac{\sqrt{\alpha_{11}}}{\sqrt{\alpha_1\alpha_{11}-\alpha_4^2}}$, $y=\frac{\sqrt[4]{\alpha_1\alpha_{11}-\alpha_4^2}}{\sqrt[4]{\alpha_{11}\alpha_6^2}}$, $q=\frac{\alpha_4( \alpha_7 +  \alpha_9)-\alpha_{11} \alpha_2}{\alpha_{6}\sqrt{\alpha_{11}(\alpha_1\alpha_{11}-\alpha_4^2)}}$, $r=\frac{(\alpha_7 \alpha_9-\alpha_{11} \alpha_5)\sqrt[4]{\alpha_1\alpha_{11}-\alpha_4^2}}{\sqrt[4]{\alpha_{11}^5\alpha_6^6}}$, $s=\frac{\alpha_9-\alpha_7}{\alpha_{6}\sqrt{\alpha_{11}}}$, $t=-\frac{\alpha_4}{\sqrt{\alpha_{11}(\alpha_1\alpha_{11}-\alpha_4^2)}}$, $u=-\frac{\alpha_9\sqrt[4]{\alpha_1\alpha_{11}-\alpha_4^2}}{\sqrt[4]{\alpha_{11}^5\alpha_6^2}}$ and $v=\frac{1}{\sqrt{\alpha_{11}}}$   , we obtain the  representative $\langle \nabla_1+\nabla_6+\nabla_{11}  \rangle$.
                    \item If $\alpha_8= \alpha_4$ and $\alpha_1\alpha_{11}=\alpha_4^2$, by choosing $x=\frac{1}{\alpha_6}$, $y=1$, $q=\frac{\alpha_4( \alpha_7 +  \alpha_9)-\alpha_{11} \alpha_2}{\alpha_{6}^2\alpha_{11}}$, $r=\frac{ \alpha_7 \alpha_9-\alpha_{11} \alpha_5}{\alpha_{11}\alpha_6}$, $s=\frac{  \alpha_9-\alpha_7}{\alpha_{6}\sqrt{\alpha_{11}}}$, $t=-\frac{\alpha_4}{\alpha_{11}\alpha_6}$, $u=-\frac{\alpha_9}{\alpha_{11}}$ and $v=\frac{1}{\sqrt{\alpha_{11}}}$, we obtain the  representative $\langle \nabla_6+\nabla_{11}  \rangle$.
                \end{enumerate}
                \item  If $\alpha_3\alpha_6\neq 0$, by choosing $x=-\frac{\alpha_6}{\alpha_3}$ and $z=y=v=1$, we return to the case $\alpha_6\neq 0$ and $\alpha_3= 0$. 
            \end{enumerate}
             \item Suppose $\alpha_{10}\neq 0$. Note that, by choosing $x=y=v=1$, $q=\frac{\alpha_{11}\alpha_3 - \alpha_{10} \alpha_8}{\alpha_{10}^2}$, $r=\frac{\alpha_{11}\alpha_6 - \alpha_{10} \alpha_9}{\alpha_{10}^2}$, $s=-\frac{\alpha_{11}}{\alpha_{10}}$, $t=-\frac{\alpha_3}{\alpha_{10}}$ and  $u=-\frac{\alpha_6}{\alpha_{10}}$, we can suppose $\alpha_3=\alpha_6=\alpha_8=\alpha_9=\alpha_{11}=0$. 
        \begin{enumerate}
           \item Suppose $\alpha_7\neq 0$. By choosing $x=y=v=1$ and $z=-\frac{\alpha_4}{\alpha_7}$ we can suppose  $\alpha_4=0$. Therefore, we have the following cases:
           \begin{enumerate}


\item If $\alpha_1(\alpha_2^2-4\alpha_1\alpha_5)\neq 0$, by choosing $x=\frac{\alpha_7}{\alpha_{10}}$, $w=-\frac{\alpha_2\alpha_7}{\alpha_{10} \sqrt{4 \alpha_1 \alpha_5-\alpha_2^2}}$, $y=\frac{2\alpha_1\alpha_7}{\alpha_{10} \sqrt{4 \alpha_1 \alpha_5-\alpha_2^2}}$ and $v=\frac{\sqrt{4 \alpha_1 \alpha_5-\alpha_2^2}}{2\alpha_{10} }$, we have the representative $\langle \nabla_1+\nabla_5+\nabla_7+\nabla_{10}  \rangle$. 
               \item If $\alpha_1\neq 0$ and $\alpha_2^2=4\alpha_1\alpha_5$, by choosing $x=\frac{\alpha_7}{\alpha_{10}}$, $w=-\frac{\alpha_2}{\alpha_{1}}$, $y=1$ and $v=\frac{\alpha_1 \alpha_7}{\alpha_{10}^2 }$, we have the representative $\langle \nabla_1+\nabla_7+\nabla_{10}  \rangle$. 
               \item If $\alpha_1= 0$ and $\alpha_2\neq 0$, by choosing $x=\frac{\alpha_7}{\alpha_{10}}$, $w=-\frac{\alpha_5}{\alpha_2}$, $y=1$ and $v=\frac{\alpha_2}{\alpha_{10}}$, we have the representative $\langle \nabla_2+\nabla_7+\nabla_{10}  \rangle$. 
               \item  If $\alpha_1=\alpha_2= 0$, we obtain the representatives $\langle \nabla_5+\nabla_7+\nabla_{10}  \rangle$ and $\langle \nabla_7+\nabla_{10}  \rangle$. 
           \end{enumerate}
           \item Suppose $\alpha_7=0$ and $\alpha_4\neq 0$. By choosing $w=z=v=1$, we return to the case $\alpha_7\neq 0$.
           \item Suppose $\alpha_7=\alpha_4=0$. If $\alpha_5\neq0$, then by choosing $w=z=v=1$ and $y=\frac{\sqrt{\alpha_2^2 - 4 \alpha_1 \alpha_5}-\alpha_2}{2 \alpha_5}$, we can suppose $\alpha_5=0$.
           \begin{enumerate}
                \item If $\alpha_2\neq 0$, by choosing  $x=y=1$, $z=-\frac{\alpha_1}{\alpha_2}$ and $v=\frac{\alpha_2}{\alpha_{10}}$, we obtain the representative $\langle \nabla_2+\nabla_{10}  \rangle$.
              \item If $\alpha_2=0$, we have the representatives $\langle \nabla_1+\nabla_{10}  \rangle$ and $\langle \nabla_{10}  \rangle$.
           \end{enumerate}
        \end{enumerate}
    \end{enumerate}
   
Summarizing, we have the following distinct orbits
\begin{center}
    $\langle \nabla_3+\nabla_7  \rangle,$
    $\langle  \nabla_3+\nabla_8  \rangle,$
    $\langle  \nabla_3+\nabla_5+\nabla_8  \rangle,$
    $\langle \nabla_3+\nabla_7+\nabla_8  \rangle,$
    $\langle \alpha\nabla_4+\nabla_6+\nabla_8  \rangle,$
    $\langle \nabla_1-\nabla_4+\nabla_6+\nabla_8  \rangle,$
    $\langle \nabla_6+\nabla_{11}  \rangle,$
    $\langle \nabla_1+\nabla_6+\nabla_{11}  \rangle,$
    $\langle \alpha\nabla_1+\nabla_4+\nabla_6+\nabla_{11}  \rangle,$
    $\langle \nabla_{10}  \rangle,$
    $\langle \nabla_1+\nabla_{10}  \rangle,$
    $\langle \nabla_2+\nabla_{10}  \rangle,$
    $\langle \nabla_7+\nabla_{10}  \rangle,$
    $\langle \nabla_1+\nabla_7+\nabla_{10}  \rangle,$
    $\langle \nabla_2+\nabla_7+\nabla_{10}  \rangle,$
    $\langle \nabla_5+\nabla_7+\nabla_{10}  \rangle,$
    $\langle \nabla_1+\nabla_5+\nabla_7+\nabla_{10}  \rangle,$
\end{center}
which give the following new algebras:

\begin{longtable}{llllllll}
$\mathcal{Z}_{06}$ &$:$& $ e_1e_2=e_3$ & $e_1 e_3 =e_5$ & $e_2e_1 =-e_3$ & $e_2 e_4 =e_5$ \\ 
\hline
$\mathcal{Z}_{07}$ &$:$& $ e_1e_2=e_3$ & $e_1e_3 =e_5$ & $e_2e_1 =-e_3$ & $e_4e_1=e_5$  \\ 
\hline
$\mathcal{Z}_{08}$ &$:$& $ e_1e_2=e_3$ & $e_1e_3 =e_5$ & $e_2e_1 =-e_3$ & $e_2e_2=e_5$ & $e_4e_1=e_5$  \\ 
\hline
$\mathcal{Z}_{09}$ &$:$& $ e_1e_2=e_3$ & $e_1 e_3 =e_5$ & $e_2e_1 =-e_3$ & $e_2 e_4 =e_5$ & $e_4e_1=e_5$ \\ 
\hline
$\mathcal{Z}_{10}^\alpha$ &$:$& $ e_1e_2=e_3$ & $e_1 e_4 =\alpha e_5$ & $e_2e_1 =-e_3$ & $e_2 e_3 =e_5$ & $e_4e_1=e_5$ \\ 
\hline
$\mathcal{Z}_{11}$ &$:$& $e_1e_1=e_5$ & $e_1e_2=e_3$ & $e_1 e_4 =-e_5$ & $e_2e_1 =-e_3$ & $e_2 e_3 =e_5$  & $e_4e_1=e_5$ \\ 
\hline
$\mathcal{Z}_{12}$ &$:$& $ e_1e_2=e_3$ & $e_2e_1 =-e_3$ & $e_2 e_3 =e_5$ & $e_4e_4=e_5$ \\ 
\hline
$\mathcal{Z}_{13}$ &$:$& $e_1e_1=e_5$ & $e_1e_2=e_3$ & $e_2e_1 =-e_3$ & $e_2 e_3 =e_5$ & $e_4e_4=e_5$ \\ 
\hline
$\mathcal{Z}_{14}^\alpha$ &$:$& $e_1e_1=\alpha e_5$ & $e_1e_2=e_3$ & $e_1e_4=e_5$ & $e_2e_1 =-e_3$ & $e_2 e_3 =e_5$  & $e_4e_4=e_5$ \\ 
\hline
$\mathcal{Z}_{15}$ &$:$& $e_1e_2=e_3$ & $e_2e_1 =-e_3$ & $e_4e_3=e_5$ \\ 
\hline
$\mathcal{Z}_{16}$ &$:$& $e_1e_1=e_5$ & $e_1e_2=e_3$ & $e_2e_1 =-e_3$ & $e_4e_3=e_5$ \\ 
\hline
$\mathcal{Z}_{17}$ &$:$& $e_1e_2=e_3+e_5$ & $e_2e_1 =-e_3$ & $e_4e_3=e_5$ \\ 
\hline
$\mathcal{Z}_{18}$ &$:$& $e_1e_2=e_3$ & $e_2e_1 =-e_3$ & $e_2e_4=e_5$ & $e_4e_3=e_5$ \\ 
\hline
$\mathcal{Z}_{19}$ &$:$& $e_1e_1=e_5$ & $e_1e_2=e_3$ & $e_2e_1 =-e_3$ & $e_2e_4=e_5$ & $e_4e_3=e_5$ \\ 
\hline
$\mathcal{Z}_{20}$ &$:$& $e_1e_2=e_3+e_5$ & $e_2e_1 =-e_3$ & $e_2e_4=e_5$ & $e_4e_3=e_5$ \\ 
\hline
$\mathcal{Z}_{21}$ &$:$& $e_1e_2=e_3$ & $e_2e_1 =-e_3$ & $e_2e_2=e_5$ & $e_2e_4=e_5$ & $e_4e_3=e_5$ \\ 
\hline
$\mathcal{Z}_{22}$ &$:$& $e_1e_1=e_5$ & $e_1e_2=e_3$ & $e_2e_1 =-e_3$ & $e_2e_2=e_5$ & $e_2e_4=e_5$  & $e_4e_3=e_5$ \\ 
\end{longtable}
\medskip

\subsubsection{Central extensions of ${\mathfrak N}_{07}$}

Let us use the following notations:
\begin{longtable}{lll}
$\nabla_1 =[\Delta_{11}],$& $\nabla_2 = [\Delta_{22}],$ & $ \nabla_3 = [\Delta_{13}-\Delta_{14}-\Delta_{23}+\Delta_{24}-2 \Delta_{32}].$
\end{longtable}
Take $\theta=\sum\limits_{i=1}^{3}\alpha_i\nabla_i\in {\rm H}^2({\mathfrak N}_{07}).$ 
	The automorphism group of ${\mathfrak N}_{07}$ is generated by invertible matrices of the form
	
\begin{center}$	\phi=
	\begin{pmatrix}
	x &  0 &  0 & 0\\
	0 &  x  &  0 & 0\\
	r &  s  &  x^2 & 0 \\
	t &  u  &  0 & x^2 \\
	\end{pmatrix}.$ \end{center}

	Since
	$$
	\phi^T\begin{pmatrix}
    \alpha_1 &  0 & \alpha_3 & -\alpha_3\\
	0  &  \alpha_2 & -\alpha_3 & \alpha_3 \\
	0 &  -2\alpha_3  & 0 & 0\\
	0 & 0 & 0 & 0
	\end{pmatrix} \phi=	\begin{pmatrix}
\alpha_1^\ast &  -\alpha^\ast & \alpha_3^\ast & -\alpha_3^\ast\\
	\alpha^{\ast\ast}  &  \alpha^\ast+\alpha_2^\ast & -\alpha_3^\ast & \alpha_3^\ast \\
	0 &  -2\alpha_3^\ast  & 0 & 0\\
	0 & 0 & 0 & 0
	\end{pmatrix},
	$$	
we have that the action of ${\rm Aut} ({\mathfrak N}_{07})$ on the subspace
$\langle \sum\limits_{i=1}^3\alpha_i\nabla_i  \rangle$
is given by
$\langle \sum\limits_{i=1}^3\alpha_i^{*}\nabla_i\rangle,$
where

\begin{longtable}{lll}
$\alpha^*_1=    \alpha_3 (r-t) x + \alpha_1x^2 ,$ &
$\alpha^*_2=   \alpha_2x^2 - 2\alpha_3 (r+s) x,$ &
$\alpha^*_3=  \alpha_3 x^3.$
\end{longtable}

Since $\alpha_3\neq 0$, we suppose that $\alpha_3=1$. Then, by choosing $x=1$, $r=\frac{\alpha_2}{2}$ and $t=\frac{2\alpha_1+\alpha_2}{2}$, we obtain the representative $\langle \nabla_3 \rangle$,
which gives the following new algebra:

\begin{longtable}{llllllll}
$\mathcal{Z}_{23}$ &$:$& $ e_1e_2=e_3$ & $e_1 e_3 =e_5$ & $e_1e_4 =-e_5$ & $e_2e_1 =e_4$ \\
&& $e_2 e_2 =-e_3$  & $e_2e_3=-e_5$ & $e_2e_4=e_5$ & $e_3e_2=-2e_5$
\end{longtable}

\subsubsection{Central extensions of ${\rm H}^2({\mathfrak N}_{08}^{1})$}

Let us use the following notations:
\begin{longtable}{ll}
$\nabla_1 =[\Delta_{11}],$ & $ \nabla_3 = [\Delta_{ 2 3}-\Delta_{ 1 3}-2\Delta_{ 3 1}+\Delta_{ 3 2}+\Delta_{4 1}],$\\
$\nabla_2 = [\Delta_{12}],$ & $ \nabla_4 = [\Delta_{ 2 4} - \Delta_{ 1 4}-\Delta_{ 3 2} - \Delta_{ 4 1}+2\Delta_{4 2}].$
\end{longtable}
Take $\theta=\sum\limits_{i=1}^{4}\alpha_i\nabla_i\in {\rm H}^2({\mathfrak N}_{08}^{1}).$ 
	The automorphism group of ${\mathfrak N}_{08}^{1}$ is generated by invertible matrices of the form
	\begin{center}
$	\phi=
	\begin{pmatrix}
	a &  x &  0 & 0\\
	a+x-y &  y  &  0 & 0\\
	z &  t  &  a(y-x) & x(y-x) \\
	u &  v  &  (a+x-y)(y-x) & y(y-x) \\
	\end{pmatrix}.$\end{center}
	

Since
	$$
	\phi^T\begin{pmatrix}
\alpha_1 &  \alpha_2 & -\alpha_3 & -\alpha_4\\
	0  &  0 & \alpha_3 & \alpha_4 \\
	-2\alpha_3 &  \alpha_3-\alpha_4  & 0 & 0\\
	\alpha_3-\alpha_4 & 2\alpha_4 & 0 & 0
	\end{pmatrix} \phi=	\begin{pmatrix}
\alpha^\ast+\alpha_1^\ast &  \alpha^{\ast\ast}+\alpha_2^\ast & -\alpha_3^\ast & -\alpha_4^\ast\\
	-\alpha^\ast  &  -\alpha^{\ast\ast} & \alpha_3^\ast & \alpha_4^\ast \\
	-2\alpha_3^\ast &  \alpha_3^\ast-\alpha_4^\ast  & 0 & 0\\
	\alpha_3^\ast-\alpha_4^\ast & 2\alpha_4^\ast & 0 & 0
	\end{pmatrix},
	$$
	 we have that the action of ${\rm Aut} ({\mathfrak N}_{08}^{1})$ on the subspace
$\langle \sum\limits_{i=1}^4\alpha_i\nabla_i  \rangle$
is given by
$\langle \sum\limits_{i=1}^4\alpha_i^{*}\nabla_i\rangle,$
where

\begin{longtable}{lll}
$\alpha^*_1$& $=$ &$  (\alpha_1 + \alpha_2)a^2 + (\alpha_1 x + 2 \alpha_2 x - \alpha_2 y + \alpha_3   (u + v-t  - z)+ \alpha_4 (  u + v -t - z))a +$ \\
   & &\multicolumn{1}{r}{$    (\alpha_2 x + \alpha_3 (t + z) - \alpha_4 (t - 2 u - 2 v + z))(x - y), $} \\
$\alpha^*_2$& $=$ &$ \alpha_1 (x^2+ax)+ \alpha_2 (x y + a  y) +\alpha_4(2  u y + 2  v y- y z- u x -  v x  -  t y )+  $\\
 && \multicolumn{1}{r}{$  \alpha_3 (  u x + v x + t y-2 t x  - 2 x z + y z), $} \\
$\alpha^*_3$& $=$ &$((\alpha_3 + \alpha_4)a  + \alpha_4 (x - y)) (x - y)^2, $ \\
$\alpha_4^*$& $=$ &$ (\alpha_3 x + \alpha_4 y)(x - y)^2 $.
\end{longtable}

Note that we can not have $(\alpha_3,\alpha_4)=(0,0)$ and by the action of ${\rm Aut} ({\mathfrak N}_{08}^{1})$, we need to consider only the following two cases:

\begin{enumerate}
    \item If $\alpha_4=0$ and $\alpha_3\neq 0$, by choosing $x=1$, $t=-\frac{\alpha_2}{\alpha_3}$ and $u=-\frac{\alpha_1+2\alpha_2}{\alpha_3}$, we obtain the representative $\langle \nabla_4 \rangle$.
    \item Suppose $\alpha_3\alpha_4\neq 0$.
    \begin{enumerate}
        \item If $\alpha_3\neq-\alpha_4$, by choosing $a=-\frac{\alpha_4}{\sqrt[3]{\alpha_3^4}}$, $x=\frac{3\alpha_3+\alpha_4}{2\sqrt[3]{\alpha_3^4}}$, 
        $y=\frac{2}{\sqrt[3]{\alpha_3}}$, $z=\frac{2\alpha_1\alpha_4-\alpha_2(\alpha_3-\alpha_4)}{\sqrt[3]{\alpha_3}(\alpha_3+\alpha_4)^2}$, $t=\frac{2(2\alpha_1\alpha_4-\alpha_2(\alpha_3-\alpha_4)}{\sqrt[3]{\alpha_3}(\alpha_3+\alpha_4)^2}$, $u=\frac{(2\alpha_2\alpha_3+\alpha_1(\alpha_3-\alpha_4))\alpha_4}{2\sqrt[3]{\alpha_3^4}(\alpha_3+\alpha_4)^2}$ and $v=\frac{(\alpha_1(\alpha_3-\alpha_4)-2\alpha_2\alpha_3)(3\alpha_3+\alpha_4)}{2\sqrt[3]{\alpha_3^4}(\alpha_3+\alpha_4)^2}$, we obtain the representative $\langle \nabla_4 \rangle$.
        \item Suppose $\alpha_3=-\alpha_4$.
        \begin{enumerate}
            \item If $\alpha_1\neq-\alpha_2$, by choosing $x=\frac{\alpha_1+\alpha_2}{\alpha_4}$ and  $v=\frac{\alpha_1(\alpha_1+\alpha_2)}{2\alpha_4^2}$, we obtain the representative $\langle \nabla_1+\nabla_3 - \nabla_4 \rangle$.
            \item If $\alpha_1=-\alpha_2$, by choosing $x=-1$ and  $v=\frac{\alpha_2}{2\alpha_4}$, we obtain the representative $\langle -\nabla_3 +\nabla_4 \rangle$.
        \end{enumerate}
    \end{enumerate}
\end{enumerate}

Summarizing, we have the distinct orbits 
\begin{center}
    $\langle \nabla_4 \rangle,$
    $\langle \nabla_1+\nabla_3 - \nabla_4 \rangle,$
    $\langle \nabla_3 -\nabla_4 \rangle$,
\end{center}
which give the following new algebras:

\begin{longtable}{llllllllll}
$\mathcal{Z}_{24}$ &$:$& $e_1e_1=e_3$ & $e_1 e_2 =e_4$ & $e_1e_4=-e_5$ & $e_2e_1 =-e_3$ & $e_2e_2 =-e_4$ \\
&& $e_2 e_4 =e_5$ & $e_3 e_2 =-e_5$ & $e_4e_1=-e_5$ & $e_4e_2=2 e_5$ \\ 
\hline
$\mathcal{Z}_{25}$ &$:$& $e_1e_1=e_3+e_5$ & $e_1 e_2 =e_4$ & $e_1e_3=-e_5$ & $e_1 e_4 =e_5$ & $e_2e_1 =-e_3$  & $e_2e_2 =-e_4$ \\
&& $e_2 e_3 =e_5$ & $e_2 e_4 =-e_5$ & $e_3 e_1 =-2e_5$ & $e_3e_2=2e_5$  & $e_4e_1=2e_5$ & $e_4e_2=-2e_5$ \\ 
\hline
$\mathcal{Z}_{26}$ &$:$& $e_1e_1=e_3$ & $e_1 e_2 =e_4$ & $e_1e_3=-e_5$ & $e_1 e_4 =e_5$  & $e_2e_1 =-e_3$  & $e_2e_2 =-e_4$ \\
&& $e_2 e_3 =e_5$ & $e_2 e_4 =-e_5$ & $e_3 e_1 =-2e_5$ & $e_3e_2=2e_5$  & $e_4e_1=2e_5$ & $e_4e_2=-2e_5$ \\ 
\end{longtable}

\subsubsection{Central extensions of ${\mathfrak N}_{12}$}

Let us use the following notations:
\begin{longtable}{llll}
$\nabla_1 =[\Delta_{11}],$& 
$\nabla_2 = [\Delta_{22}],$ & 
$\nabla_3 = [\Delta_{14}-\Delta_{13}],$ &
$\nabla_4 = [\Delta_{24}-\Delta_{23}].$ 
\end{longtable}
Take $\theta=\sum\limits_{i=1}^4\alpha_i\nabla_i\in {\rm H^2}({\mathfrak N}_{12}).$ 
	The automorphism group of ${\mathfrak N}_{12}$ is generated by invertible matrices of the form
\begin{center}$	\phi_1=
	\begin{pmatrix}
	0 &  x &  0 & 0\\
	y &  0  &  0 & 0\\
	z &  t  &  0 & xy \\
	u &  v  &  xy & 0 \\
	\end{pmatrix}$ and $	\phi_2=
	\begin{pmatrix}
	x &    0  &  0 & 0\\
	0 &  y  &  0 & 0\\
	z &  t  &  xy & 0 \\
	u &  v  &  0 & xy \\
	\end{pmatrix}.$\end{center}
	Since
	$$
	\phi^T_1\begin{pmatrix}
\alpha_1 &  0 & -\alpha_3 & \alpha_3\\
	0  &  \alpha_2 & -\alpha_4 & \alpha_4 \\
	0&  0  & 0 & 0\\
	0 & 0 & 0 & 0
	\end{pmatrix} \phi_1=	\begin{pmatrix}
\alpha_1^* &  \alpha^* & -\alpha_3^* & \alpha_3^*\\
	\beta^*  &  \alpha_2^* & -\alpha_4^* & \alpha_4^* \\
	0 &  0 & 0 & 0\\
	0 & 0 & 0 & 0
	\end{pmatrix},
	$$
	 we have that the action of ${\rm Aut} ({\mathfrak N}_{12})$ on the subspace
$\langle \sum\limits_{i=1}^4\alpha_i\nabla_i  \rangle$
is given by
$\langle \sum\limits_{i=1}^4\alpha_i^{*}\nabla_i\rangle,$
where
\begin{longtable}{llll}
$\alpha^*_1=   ( \alpha_2y + \alpha_4u-\alpha_4z)y,$ &
$\alpha^*_2=   (\alpha_1x-\alpha_3t+\alpha_3v)x,$ &
$\alpha^*_3=  -\alpha_4xy^2,$ &
$\alpha_4^*=  -\alpha_3x^2y.$\\
\end{longtable}

Since ${\rm H}^2({\mathfrak N}_{12})=
\Big\langle 
[\Delta_{ 1 1}], [\Delta_{ 2 2}],[\Delta_{ 1 4}-\Delta_{ 1 3}], [\Delta_{ 2 4}-\Delta_{ 2 3}]
\Big\rangle $ and we are interested only in new algebras, we must have $(\alpha_3,\alpha_4)\neq (0,0)$.
Then 
\begin{enumerate}
    \item if $\alpha_4\alpha_3 \neq 0$, then by choosing 
    $x=\alpha_4$, $y=\alpha_3$, $z=\frac{\alpha_2\alpha_3}{\alpha_4}$ and $t=\frac{\alpha_1 \alpha_4}{\alpha_3}$,         we have the representative $\langle \nabla_3+ \nabla_4 \rangle.$
    
    \item if $\alpha_4\alpha_1\neq 0$ and $\alpha_3=0$, then by choosing 
    $x=-\alpha_1$, $y=\frac{\alpha_1}{\sqrt{\alpha_4}}$ and $z=\frac{\alpha_2y}{\alpha_4}$,         we have the representative $\langle \nabla_2+ \nabla_3 \rangle.$
    
    \item if $\alpha_4\neq 0$ and $\alpha_1=\alpha_3=0$, then by choosing 
    $x=y=1$ and $z=\frac{\alpha_2}{\alpha_4}$,         we have the representative $\langle \nabla_3 \rangle.$
    
   \item if $\alpha_3\neq 0$ and $\alpha_4=0$, by applying $\phi_2$ we obtain the case $\alpha_4^*=0$ which has been considered above.
\end{enumerate}
Summarizing, we have the following distinct orbits 
\begin{center} 
$\langle \nabla_3+\nabla_4 \rangle,$ \  
$\langle  \nabla_2+ \nabla_3 \rangle,$ \  
$\langle \nabla_3 \rangle$,
\end{center}
which give the following new algebras:

\begin{longtable}{llllllll}
$\mathcal{Z}_{27}$ &$:$& $ e_1e_2=e_3$ & $e_1 e_3 =-e_5$ & $e_1 e_4 =e_5$ & $e_2e_1=e_4$ & $e_2e_3 =-e_5$ & $e_2e_4=e_5$\\
\hline
$\mathcal{Z}_{28}$ &$:$& $ e_1e_2=e_3$ & $e_1 e_3 =-e_5$ & $e_1 e_4 =e_5$ & $e_2e_1=e_4$ & $e_2e_2 =e_5$ \\
\hline
$\mathcal{Z}_{29}$ &$:$& $ e_1e_2=e_3$ & $e_1 e_3 =-e_5$ & $e_1 e_4 =e_5$ & $e_2e_1=e_4$ &   \\
\end{longtable}

\subsubsection{Central extensions of ${\mathfrak N}_{14}^{\alpha\neq-1}$}

Let us use the following notations:
\begin{longtable}{llll}
$\nabla_1 =[\Delta_{11}],$& $\nabla_2 = [\Delta_{21}],$ &
$\nabla_3=[\Delta_{23}+2\Delta_{32}],$ & 
$\nabla_4 = [2\alpha\Delta_{ 2 4}+(\alpha+1)\left(\Delta_{1 3}+2\alpha\Delta_{3 1}+2\Delta_{4 2}\right)].$ 
\end{longtable}
Take $\theta=\sum\limits_{i=1}^4\alpha_i\nabla_i\in {\rm H^2}({\mathfrak N}_{14}^{\alpha\neq-1}).$ 
	The automorphism group of ${\mathfrak N}_{14}^{\alpha\neq-1}$ consists of invertible matrices of the form

\begin{center}$	\phi=
	\begin{pmatrix}
	x &  y &  0 & 0\\
	0 &  z  &  0 & 0\\
	t &  u  &  z^2 & 0 \\
	v & w  &  yz(\alpha+1) & xz \\
	\end{pmatrix}$.\end{center}
	Since
{\tiny	$$
	\phi^T\begin{pmatrix}
\alpha_1 &  0 & (\alpha+1)\alpha_4 & 0\\
\alpha_2  & 0 & \alpha_3 & 2\alpha\alpha_4 \\
2\alpha(\alpha+1)\alpha_4 &  2\alpha_3  & 0 & 0\\
	0 & 2(\alpha+1)\alpha_4 & 0 & 0
	\end{pmatrix} \phi=	\begin{pmatrix}
\alpha_1^* &  \beta^* & (\alpha+1)\alpha_4^* & 0\\
\alpha_2^*+\alpha\beta^*  & \gamma^* & \alpha_3^* & 2\alpha\alpha_4^* \\
2\alpha(\alpha+1)\alpha_4^* &  2\alpha_3^*  & 0 & 0\\
	0 & 2(\alpha+1)\alpha_4^* & 0 & 0
	\end{pmatrix},
	$$}
	 we have that the action of ${\rm Aut} ({\mathfrak N}_{14}^{\alpha\neq-1})$ on the subspace
$\langle \sum\limits_{i=1}^4\alpha_i\nabla_i  \rangle$
is given by
$\langle \sum\limits_{i=1}^4\alpha_i^{*}\nabla_i\rangle,$
where
\begin{longtable}{lcl}
$\alpha^*_1$&$=$&$  (\alpha_1x+(2\alpha^2+3\alpha+1)\alpha_4t)x,$ \\
$\alpha^*_2$&$=$&$   (\alpha+1)(1-2\alpha^2)\alpha_4ty+\alpha(\alpha+1)\alpha_4ux+$\\
&& \multicolumn{1}{r}{$(1-\alpha)\alpha_1xy-2\alpha^2\alpha_4vz+\alpha_2xz+(1-2\alpha)\alpha_3tz,$}\\
$\alpha^*_3$&$=$&$  (\alpha_3z+(2\alpha^2+3\alpha+1)\alpha_4y)z^2,$\\
$\alpha_4^*$&$=$&$  \alpha_4xz^2.$\\
\end{longtable}

Since ${\rm H}^2({\mathfrak N}_{14}^{\alpha\neq-1})=
\Big\langle 
[\Delta_{ 1 1}], [\Delta_{ 2 1}],[\Delta_{ 2 3}+2\Delta_{ 3 2}], [2\alpha\Delta_{ 2 4}+(\alpha+1)\left(\Delta_{1 3}+2\alpha\Delta_{3 1}+2\Delta_{4 2}\right)]
\Big\rangle $ and we are interested only in new algebras, we must have $\alpha_4\neq 0$.
Then

\begin{enumerate}
    \item if $\alpha\neq0,-\frac{1}{2}$, 
    then by choosing 
    $x=z=-\alpha_4(2\alpha^2+3\alpha+1)$, $y=\alpha_3$, $t=\alpha_1$ and $v=-\frac{\alpha_2\alpha_4(4\alpha^3+8\alpha^2+5\alpha+1)+\alpha_1\alpha_3(4\alpha^2-\alpha-1)}{2\alpha^2(2\alpha+1)\alpha_4}$,         we have the representative $\langle \nabla_4 \rangle.$
    
    \item if $\alpha=-\frac{1}{2}$, 
    \begin{enumerate}
        \item Suppose $\alpha_3= 0$.
        \begin{enumerate}
            \item If $\alpha_1=0$, then by choosing $x=z=1$ and $u=\frac{4\alpha_2}{\alpha_4}$, we obtain the representative $\langle \nabla_4 \rangle.$
            \item If $\alpha_1\neq 0$, then by choosing $x=\alpha_1\alpha_4$, $z=\alpha_1$ and $u=\frac{4\alpha_1\alpha_2}{\alpha_4}$, we obtain the representative $\langle \nabla_1 + \nabla_4 \rangle.$
        \end{enumerate}
      
  \item Suppose $\alpha_3\neq 0.$

         \begin{enumerate}
            \item If $\alpha_1=0$, then by choosing $x=\alpha_3$, $z=\alpha_4$ and $u=4\alpha_2$, we obtain the representative $\langle \nabla_3 + \nabla_4 \rangle.$
       \item If $\alpha_1\neq 0$, then by choosing 
     $x=\frac{\alpha_1\alpha_3^2}{\alpha_4^3}$, 
     $z=\frac{\alpha_1\alpha_3}{\alpha_4^2}$ and 
     $u=\frac{4\alpha_1\alpha_2\alpha_3}{\alpha_4^3}$, we obtain the representative $\langle \nabla_1+\nabla_3 + \nabla_4 \rangle.$
        \end{enumerate}
   \end{enumerate}
    
    \item if $\alpha=0$ and $\alpha_2\alpha_4-\alpha_1\alpha_3\neq 0$, then by choosing 
    $x=\alpha_4$, $y=\frac{\alpha_3(\alpha_1\alpha_3-\alpha_2\alpha_4)}{\alpha_4^3}$, $z=\frac{\alpha_2\alpha_4-\alpha_1\alpha_3}{\alpha_4^2}$ and $t=-\alpha_1$,         we have the representative $\langle \nabla_2+\nabla_4 \rangle.$
    
   \item if $\alpha=0$ and $\alpha_2\alpha_4-\alpha_1\alpha_3=0$, then by choosing 
    $x=z=\alpha_4$, $y=-\alpha_3$ and $t=-\alpha_1$,         we have the representative $\langle \nabla_4 \rangle.$
\end{enumerate}
Summarizing, we have 
\begin{center}$\langle \nabla_4\rangle_{\alpha\neq -1}$,
$\langle \nabla_2+\nabla_4\rangle_{\alpha=0}$,
$\langle \nabla_1+\nabla_4\rangle_{\alpha=-\frac{1}{2}}$,
$\langle \nabla_3+\nabla_4\rangle_{\alpha=-\frac{1}{2}}$,
$\langle \nabla_1+\nabla_3+\nabla_4\rangle_{\alpha=-\frac{1}{2}}$,
\end{center}
which give the following new algebras:

\begin{longtable}{llllllll}
$\mathcal{Z}_{30}^{\alpha\neq-1}$ &$:$& $ e_1e_2=e_4$ & $e_1 e_3 =(\alpha+1)e_5$ & $e_2 e_1 =\alpha e_4$ & $e_2e_2=e_3$ \\
&& $e_2e_4 =2\alpha e_5$  
 & 
$e_3e_1=2\alpha(\alpha+1)e_5$ & 
\multicolumn{1}{l}{$e_4e_2=2(\alpha+1)e_5$}\\
\hline

$\mathcal{Z}_{31}$ &$:$& $ e_1e_2=e_4$ & $e_1 e_3 =e_5$ & $e_2 e_1 =e_5$ & $e_2e_2=e_3$ &  $e_4e_2=2e_5$ \\
\hline

$\mathcal{Z}_{32}$ &$:$& $e_1e_1=e_5$ & $ e_1e_2=e_4$ & $e_1 e_3 =\frac{1}{2}e_5$ & $e_2 e_1 =-\frac{1}{2}e_4$ \\
&& $e_2e_2=e_3$ & $e_2e_4 =-e_5$ & $e_3e_1=-\frac{1}{2}e_5$ & $e_4e_2=e_5$\\
\hline

$\mathcal{Z}_{33}$ &$:$& $ e_1e_2=e_4$ & $e_1 e_3 =\frac{1}{2}e_5$ & $e_2 e_1 =-\frac{1}{2}e_4$ & $e_2e_2=e_3$ & $e_2e_3=e_5$\\
&& $e_2e_4 =-e_5$ & $e_3e_1=-\frac{1}{2}e_5$ & $e_3e_2=2e_5$ & $e_4e_2=e_5$\\
\hline
$\mathcal{Z}_{34}$ &$:$& $e_1e_1=e_5$ & $e_1e_2=e_4$ & $e_1 e_3 =\frac{1}{2}e_5$ & $e_2 e_1 =-\frac{1}{2}e_4$ & $e_2e_2=e_3$ \\
&& $e_2e_3=e_5$ & $e_2e_4 =-e_5$ & $e_3e_1=-\frac{1}{2}e_5$ & $e_3e_2=2e_5$ & $e_4e_2=e_5$\\
\end{longtable}


\subsubsection{Central extensions of ${\mathfrak N}_{14}^{-1}$}

Let us use the following notations:
\begin{longtable}{lllll}
$\nabla_1 =[\Delta_{11}],$& 
$\nabla_2 = [\Delta_{14}],$ & 
$\nabla_3=[\Delta_{21}],$ &
$ \nabla_4 = [\Delta_{ 2 4}],$ & 
$ \nabla_5= [\Delta_{ 2 3}+2\Delta_{32}].$ 
\end{longtable}
Take $\theta=\sum\limits_{i=1}^5\alpha_i\nabla_i\in {\rm H^2}({\mathfrak N}_{14}^{-1}).$ 
	The automorphism group of ${\mathfrak N}_{14}^{-1}$ consists of invertible matrices of the form

\begin{center}$	\phi=
	\begin{pmatrix}
	x &  y &  0 & 0\\
	0 &  z  &  0 & 0\\
	t &  u  &  z^2 & 0 \\
	v & w  &  0 & xz \\
	\end{pmatrix}$.\end{center}
	Since
	$$
	\phi^T\begin{pmatrix}
\alpha_1 &  0 & 0 & \alpha_2\\
\alpha_3 & 0 & \alpha_5 & \alpha_4 \\
0 &  2\alpha_5  & 0 & 0\\
	0 & 0 & 0 & 0
	\end{pmatrix} \phi=	\begin{pmatrix}
\alpha_1^* &  \beta^* & 0 & \alpha_2^*\\
\alpha_3^*-\beta^*  & \gamma^* & \alpha_5^* & \alpha_4^* \\
0 &  2\alpha_5^*  & 0 & 0\\
	0 & 0 & 0 & 0
	\end{pmatrix},
	$$
	 we have that the action of ${\rm Aut} ({\mathfrak N}_{14}^{-1})$ on the subspace
$\langle \sum\limits_{i=1}^5\alpha_i\nabla_i  \rangle$
is given by
$\langle \sum\limits_{i=1}^5\alpha_i^{*}\nabla_i\rangle,$
where
\begin{longtable}{l}
$\alpha^*_1=  (\alpha_1x+\alpha_2v)x,$ \\
$\alpha^*_2=   \alpha_2x^2z,$ \\
$\alpha^*_3=  2\alpha_1xy+\alpha_2vy+\alpha_2wx+\alpha_3xz+\alpha_4vz+3\alpha_5tz,$\\
$\alpha_4^*=(\alpha_2y+\alpha_4z)xz,$\\
$\alpha^*_5=   \alpha_5z^3.$ \\
\end{longtable}

Since ${\rm H}^2({\mathfrak N}_{14}^{-1})=
\Big\langle 
[\Delta_{ 1 1}], [\Delta_{ 2 1}]
, [\Delta_{ 1 4}], [\Delta_{ 2 4}], [\Delta_{ 2 3}+2\Delta_{ 3 2}]
\Big\rangle $ and we are interested only in new algebras, we must have $\alpha_5\neq 0$ and $(\alpha_2,\alpha_4)\neq (0,0)$.
Then:

\begin{enumerate}
    \item If $\alpha_2\neq 0$, then by choosing 
    $x=\sqrt{\alpha_2\alpha_5}$, $y=-\alpha_4$, $z=\alpha_2$, $v=-\frac{\alpha_1\sqrt{\alpha_5}}{\sqrt{\alpha_2}}$ and $w=\frac{2\alpha_1\alpha_4-\alpha_2\alpha_3}{\alpha_2}$,         we have the representative $\langle \nabla_2+\nabla_5 \rangle.$
    
    \item If $\alpha_2=0$ and $\alpha_4\alpha_1\neq 0$, then by choosing 
    $x=\frac{\alpha_1\alpha_5^2}{\alpha_4^3}$, $z=\frac{\alpha_1\alpha_5}{\alpha_4^2}$ and $t=-\frac{\alpha_1\alpha_3\alpha_5}{3\alpha_4^2}$,         we have the representative $\langle \nabla_1+\nabla_4+\nabla_5 \rangle.$
    
    \item If $\alpha_4\neq 0$ and $\alpha_2=\alpha_1=0$, then by choosing 
    $x=\alpha_5$, $z=\alpha_4$ and $t=-\frac{\alpha_3}{3}$,         we have the representative $\langle \nabla_4+\nabla_5 \rangle.$
\end{enumerate}
Summarizing, we have the following distinct orbits: 
\begin{center} 
$\langle \nabla_2+\nabla_5 \rangle,$ \  
$\langle  \nabla_1+ \nabla_4+\nabla_5 \rangle,$ \  
$\langle \nabla_4+\nabla_5 \rangle,$
\end{center}
which give the following new algebras:

\begin{longtable}{lllllllllll}
$\mathcal{Z}_{35}$ &$:$& $ e_1e_2=e_4$ & $e_1 e_4 =e_5$ & $e_2 e_1 =-e_4$ 
& $e_2e_2=e_3$ &  $e_2e_3=e_5$ & $e_3 e_2 =2e_5$\\
\hline
$\mathcal{Z}_{36}$ &$:$& $e_1e_1=e_5$ & $ e_1e_2=e_4$ & $e_2 e_1 =-e_4$ & $e_2e_2=e_3$ 
&  $e_2e_3=e_5$  & $e_2 e_4 =e_5$ & $e_3 e_2 =2e_5$  \\
\hline
$\mathcal{Z}_{37}$ &$:$& $ e_1e_2=e_4$ & $e_2 e_1 =-e_4$ & $e_2e_2=e_3$ 
&  $e_2e_3=e_5$ & $e_2 e_4 =e_5$ & $e_3 e_2 =2e_5$ \\
\end{longtable}

\subsubsection{Central extensions of ${\mathfrak Z}_{1}$} 

Let us use the following notations:
\begin{longtable}{llll}
$\nabla_1 =[\Delta_{13}+3\Delta_{22}+3\Delta_{31}],$& 
$\nabla_2 = [\Delta_{14}],$ & 
$ \nabla_3 = [\Delta_{41}],$ &
$\nabla_4 = [\Delta_{44}].$ 
\end{longtable}
Take $\theta=\sum\limits_{i=1}^4\alpha_i\nabla_i\in {\rm H^2}({\mathfrak Z}_{1}).$ 
	The automorphism group of ${\mathfrak Z}_{1}$ is generated by invertible matrices of the form
\begin{center}$	\phi=
	\begin{pmatrix}
	x &  0 &  0 & 0\\
	y &  x^2  &  0 & 0\\
	z &  3xy  &  x^3 & t \\
	u &  0  &  0 & v \\
	\end{pmatrix}.$\end{center}
	Since
	$$
	\phi^T\begin{pmatrix}
0 &  0 & \alpha_1 & \alpha_2\\
0 & 3\alpha_1 & 0 & 0 \\
3\alpha_1 &  0  & 0 & 0\\
\alpha_3 & 0 & 0 & \alpha_4
	\end{pmatrix} \phi=	\begin{pmatrix}
\alpha^* & \beta^* & \alpha_1^* & \alpha_2^*\\
	2\beta^*  &  3\alpha_1^* & 0 & 0 \\
	3\alpha_1^* &  0 & 0 & 0\\
	\alpha_3^* & 0 & 0 & \alpha_4^*
	\end{pmatrix},
	$$
	 we have that the action of ${\rm Aut} ({\mathfrak Z}_{1})$ on the subspace
$\langle \sum\limits_{i=1}^4\alpha_i\nabla_i  \rangle$
is given by
$\langle \sum\limits_{i=1}^4\alpha_i^{*}\nabla_i\rangle,$
where
\begin{longtable}{ll}
$\alpha^*_1=  \alpha_1x^4,$ &
$\alpha^*_2= \alpha_1xt+(\alpha_2x+\alpha_4u)v,$ \\
$\alpha^*_3=  (3\alpha_1t+\alpha_3v)x+\alpha_4vu,$ &
$\alpha_4^*=  \alpha_4v^2.$\\
\end{longtable}

Since ${\rm H}^2({\mathfrak Z}_{1})=
\Big\langle 
[\Delta_{ 1 3}+3\Delta_{ 22}+3\Delta_{31}], [\Delta_{ 1 4}],[\Delta_{ 41}], [\Delta_{44}]
\Big\rangle $
and we are interested only in new algebras, we must have $\alpha_1\neq0$ and $(\alpha_2,\alpha_3,\alpha_4)\neq (0,0,0)$.
Then 
\begin{enumerate}
    \item If $\alpha_4\neq 0$, then by taking $x=1$, $t=\frac{\alpha_2-\alpha_3}{2\sqrt{\alpha_1\alpha_4}}$, $u=\frac{\alpha_3-3\alpha_2}{2\alpha_4}$ and $v=\sqrt{\frac{\alpha_1}{\alpha_4}}$, we obtain the representative $\langle \nabla_1+\nabla_4\rangle$.
    
    \item If $\alpha_4=0$ and $\alpha_3\neq 3\alpha_2$, then by taking $x=1$, $t=\frac{\alpha_3}{\alpha_3-3\alpha_2}$ and $v=\frac{3\alpha_1}{3\alpha_2-\alpha_3}$, we obtain the representative $\langle \nabla_1+\nabla_2\rangle$.
    
    \item If $\alpha_4=0$ and $\alpha_3=3\alpha_2$, then by taking $x=v=1$ and $t=-\frac{\alpha_2}{\alpha_1}$, then we obtain the representative $\langle \nabla_1\rangle$, which does not satisfy condition $(\alpha_2,\alpha_3,\alpha_4)\neq 0$. So we will not consider it.
\end{enumerate}

Summarizing, we have the following distinct orbits: 
\begin{center} 
$\langle \nabla_1+\nabla_4 \rangle,$ \  
$\langle  \nabla_1+ \nabla_2 \rangle,$ \  
\end{center}
which give the following new algebras:

\begin{longtable}{lllllllllllll}
$\mathcal{Z}_{38}$ &$:$& $ e_1e_1=e_2$ & $e_1 e_2 =e_3$ & $e_1e_3=e_5$ & $e_2e_1=2e_3$  & $e_2e_2=3e_5$  & $e_3e_1=3e_5$ & $e_4e_4=e_5$ \\
\hline
$\mathcal{Z}_{39}$ &$:$& $ e_1e_1=e_2$ & $e_1 e_2 =e_3$ & $e_1e_3=e_5$ & $e_1e_4=e_5$  & $e_2e_1=2e_3$  & $e_2e_2=3e_5$ & $e_3e_1=3e_5$  \\
\end{longtable}

\subsection{Classification theorem for $5$-dimensional Zinbiel algebras}\label{secteoA}
Thanks to \cite{dzhuma5} each complex finite-dimensional Zinbiel algebra is nilpotent.
Hence,
the algebraic classification of complex $5$-dimensional Zinbiel algebras consists of two parts:
\begin{enumerate}
    \item $5$-dimensional algebras with identity $xyz=0$ (also known as $2$-step nilpotent algebras) is the intersection of all varieties of algebras defined by a family of polynomial identities   of degree three or more; for example, it is in intersection of associative, Zinbiel, Leibniz, etc, algebras. All these algebras can be obtained as central extensions of zero-product algebras. The geometric classification of $2$-step nilpotent algebras is given in \cite{ikp20}. It is the reason why we are not interested in it.
    
    \item $5$-dimensional Zinbiel (non-$2$-step nilpotent) algebras, which are central extensions of  Zinbiel  algebras with nonzero product of a smaller dimension. These algebras are classified by several steps:
    \begin{enumerate}
        \item split complex $5$-dimensional Zinbiel algebras are classified in  \cite{dzhuma5,centr3zinb};
        
        \item non-split  complex $5$-dimensional Zinbiel algebras with $2$-dimensional annihilator are classified in \cite{centr3zinb};

        \item non-split  complex $5$-dimensional Zinbiel algebras with $1$-dimensional annihilator are classified in Theorem A (see below).
            \end{enumerate}
  \end{enumerate}

\begin{theoremA}
Let ${\mathcal Z}$ be a non-split complex  $5$-dimensional  Zinbiel (non-$2$-step nilpotent) algebra.
Then ${\mathcal Z}$ is isomorphic to one algebra from the following list:

 \begin{longtable}{lllllll}

$\mathcal{Z}_{01}$ &$:$& $ e_1e_1=e_2$ & $e_1 e_2 =e_5$ & $e_1 e_3 =e_5$ & $e_2e_1=2e_5$ & $e_4e_4 =e_5$ \\

$\mathcal{Z}_{02}^\alpha$ &$:$& $ e_1e_1=e_2$ & $e_1 e_2 =e_5$ & $e_2e_1=2e_5$ & $e_3e_4=e_5$ & $e_4e_3 =\alpha e_5$ \\

$\mathcal{Z}_{03}$ &$:$& $ e_1e_1=e_2$ & $e_1 e_2 =e_5$ & $e_2e_1=2e_5$ \\
&& $e_3e_3=e_5$ & $e_3e_4=e_5$ & $e_4e_3 =-e_5$ \\

$\mathcal{Z}_{04}$ &$:$& $ e_1e_1=e_2$ & $e_1 e_2 =e_5$ & $e_1e_4=e_5$ \\
&& $e_2e_1=2e_5$ & $e_3e_4=e_5$ & $e_4e_3=2e_5$ \\


$\mathcal{Z}_{05}$ &$:$& $ e_1e_1=e_3$ & $e_1 e_3 =e_5$ & $e_2e_2 =e_4$ \\
&& $e_2 e_4 =e_5$ & $e_3e_1=2e_5$ & $e_4e_2=2e_5$ \\


$\mathcal{Z}_{06}$ &$:$& $ e_1e_2=e_3$ & $e_1 e_3 =e_5$ & $e_2e_1 =-e_3$ & $e_2 e_4 =e_5$ \\ 

$\mathcal{Z}_{07}$ &$:$& $ e_1e_2=e_3$ & $e_1e_3 =e_5$ & $e_2e_1 =-e_3$ & $e_4e_1=e_5$  \\ 

$\mathcal{Z}_{08}$ &$:$& $ e_1e_2=e_3$ & $e_1e_3 =e_5$ & $e_2e_1 =-e_3$ & $e_2e_2=e_5$ & $e_4e_1=e_5$  \\ 

$\mathcal{Z}_{09}$ &$:$& $ e_1e_2=e_3$ & $e_1 e_3 =e_5$ & $e_2e_1 =-e_3$ & $e_2 e_4 =e_5$ & $e_4e_1=e_5$ \\ 

$\mathcal{Z}_{10}^\alpha$ &$:$& $ e_1e_2=e_3$ & $e_1 e_4 =\alpha e_5$ & $e_2e_1 =-e_3$ & $e_2 e_3 =e_5$ & $e_4e_1=e_5$ \\ 

$\mathcal{Z}_{11}$ &$:$& $e_1e_1=e_5$ & $e_1e_2=e_3$ & $e_1 e_4 =-e_5$ \\
&& $e_2e_1 =-e_3$ & $e_2 e_3 =e_5$  & $e_4e_1=e_5$ \\ 

$\mathcal{Z}_{12}$ &$:$& $ e_1e_2=e_3$ & $e_2e_1 =-e_3$ & $e_2 e_3 =e_5$ & $e_4e_4=e_5$ \\ 

$\mathcal{Z}_{13}$ &$:$& $e_1e_1=e_5$ & $e_1e_2=e_3$ & $e_2e_1 =-e_3$ & $e_2 e_3 =e_5$ & $e_4e_4=e_5$ \\ 

$\mathcal{Z}_{14}^\alpha$ &$:$& $e_1e_1=\alpha e_5$ & $e_1e_2=e_3$ & $e_1e_4=e_5$ \\
&& $e_2e_1 =-e_3$ & $e_2 e_3 =e_5$  & $e_4e_4=e_5$ \\ 

$\mathcal{Z}_{15}$ &$:$& $e_1e_2=e_3$ & $e_2e_1 =-e_3$ & $e_4e_3=e_5$ \\ 

$\mathcal{Z}_{16}$ &$:$& $e_1e_1=e_5$ & $e_1e_2=e_3$ & $e_2e_1 =-e_3$ & $e_4e_3=e_5$ \\ 

$\mathcal{Z}_{17}$ &$:$& $e_1e_2=e_3+e_5$ & $e_2e_1 =-e_3$ & $e_4e_3=e_5$ \\ 

$\mathcal{Z}_{18}$ &$:$& $e_1e_2=e_3$ & $e_2e_1 =-e_3$ & $e_2e_4=e_5$ & $e_4e_3=e_5$ \\ 

$\mathcal{Z}_{19}$ &$:$& $e_1e_1=e_5$ & $e_1e_2=e_3$ & $e_2e_1 =-e_3$ & $e_2e_4=e_5$ & $e_4e_3=e_5$ \\ 

$\mathcal{Z}_{20}$ &$:$& $e_1e_2=e_3+e_5$ & $e_2e_1 =-e_3$ & $e_2e_4=e_5$ & $e_4e_3=e_5$ \\ 

$\mathcal{Z}_{21}$ &$:$& $e_1e_2=e_3$ & $e_2e_1 =-e_3$ & $e_2e_2=e_5$ & $e_2e_4=e_5$ & $e_4e_3=e_5$ \\ 

$\mathcal{Z}_{22}$ &$:$& $e_1e_1=e_5$ & $e_1e_2=e_3$ & $e_2e_1 =-e_3$ \\
&& $e_2e_2=e_5$ & $e_2e_4=e_5$  & $e_4e_3=e_5$ \\ 


$\mathcal{Z}_{23}$ &$:$& $ e_1e_2=e_3$ & $e_1 e_3 =e_5$ & $e_1e_4 =-e_5$ & $e_2e_1 =e_4$ \\
&& $e_2 e_2 =-e_3$  & $e_2e_3=-e_5$ & $e_2e_4=e_5$ & $e_3e_2=-2e_5$ \\


$\mathcal{Z}_{24}$ &$:$& $e_1e_1=e_3$ & $e_1 e_2 =e_4$ & $e_1e_4=-e_5$ & $e_2e_1 =-e_3$ & $e_2e_2 =-e_4$ \\
&& $e_2 e_4 =e_5$ & $e_3 e_2 =-e_5$ & $e_4e_1=-e_5$ & $e_4e_2=2 e_5$ \\ 

$\mathcal{Z}_{25}$ &$:$& $e_1e_1=e_3+e_5$ & $e_1 e_2 =e_4$ & $e_1e_3=-e_5$ & $e_1 e_4 =e_5$ \\
&& $e_2e_1 =-e_3$  & $e_2e_2 =-e_4$ & $e_2 e_3 =e_5$ & $e_2 e_4 =-e_5$ \\
&& $e_3 e_1 =-2e_5$ & $e_3e_2=2e_5$  & $e_4e_1=2e_5$ & $e_4e_2=-2e_5$ \\ 

$\mathcal{Z}_{26}$ &$:$& $e_1e_1=e_3$ & $e_1 e_2 =e_4$ & $e_1e_3=-e_5$ & $e_1 e_4 =e_5$ \\
&& $e_2e_1 =-e_3$  & $e_2e_2 =-e_4$ & $e_2 e_3 =e_5$ & $e_2 e_4 =-e_5$ \\
&& $e_3 e_1 =-2e_5$ & $e_3e_2=2e_5$  & $e_4e_1=2e_5$ & $e_4e_2=-2e_5$ \\ 


$\mathcal{Z}_{27}$ &$:$& $ e_1e_2=e_3$ & $e_1 e_3 =-e_5$ & $e_1 e_4 =e_5$ \\
&& $e_2e_1=e_4$ & $e_2e_3 =-e_5$ & $e_2e_4=e_5$\\

$\mathcal{Z}_{28}$ &$:$& $ e_1e_2=e_3$ & $e_1 e_3 =-e_5$ & $e_1 e_4 =e_5$ & $e_2e_1=e_4$ & $e_2e_2 =e_5$ \\

$\mathcal{Z}_{29}$ &$:$& $ e_1e_2=e_3$ & $e_1 e_3 =-e_5$ & $e_1 e_4 =e_5$ & $e_2e_1=e_4$ &   \\


$\mathcal{Z}_{30}^{\alpha\neq-1}$ &$:$& $ e_1e_2=e_4$ & $e_1 e_3 =(\alpha+1)e_5$ & $e_2 e_1 =\alpha e_4$ & $e_2e_2=e_3$ \\
&& $e_2e_4 =2\alpha e_5$  
 & 
\multicolumn{2}{l}{$e_3e_1=2\alpha(\alpha+1)e_5$} & 
\multicolumn{2}{l}{$e_4e_2=2(\alpha+1)e_5$}\\


$\mathcal{Z}_{31}$ &$:$& $ e_1e_2=e_4$ & $e_1 e_3 =e_5$ & $e_2 e_1 =e_5$ & $e_2e_2=e_3$ &  $e_4e_2=2e_5$ \\


$\mathcal{Z}_{32}$ &$:$& $e_1e_1=e_5$ & $ e_1e_2=e_4$ & $e_1 e_3 =\frac{1}{2}e_5$ & $e_2 e_1 =-\frac{1}{2}e_4$ \\
&& $e_2e_2=e_3$ & $e_2e_4 =-e_5$ & $e_3e_1=-\frac{1}{2}e_5$ & $e_4e_2=e_5$\\

$\mathcal{Z}_{33}$ &$:$& $ e_1e_2=e_4$ & $e_1 e_3 =\frac{1}{2}e_5$ & $e_2 e_1 =-\frac{1}{2}e_4$ & $e_2e_2=e_3$ & $e_2e_3=e_5$\\
&& $e_2e_4 =-e_5$ & $e_3e_1=-\frac{1}{2}e_5$ & $e_3e_2=2e_5$ & $e_4e_2=e_5$\\

$\mathcal{Z}_{34}$ &$:$& $e_1e_1=e_5$ & $e_1e_2=e_4$ & $e_1 e_3 =\frac{1}{2}e_5$ & $e_2 e_1 =-\frac{1}{2}e_4$ & $e_2e_2=e_3$ \\
&& $e_2e_3=e_5$ & $e_2e_4 =-e_5$ & $e_3e_1=-\frac{1}{2}e_5$ & $e_3e_2=2e_5$ & $e_4e_2=e_5$\\


$\mathcal{Z}_{35}$ &$:$& $ e_1e_2=e_4$ & $e_1 e_4 =e_5$ & $e_2 e_1 =-e_4$ \\
&& $e_2e_2=e_3$ &  $e_2e_3=e_5$ & $e_3 e_2 =2e_5$\\

$\mathcal{Z}_{36}$ &$:$& $e_1e_1=e_5$ & $ e_1e_2=e_4$ & $e_2 e_1 =-e_4$ & $e_2e_2=e_3$ \\
&&  $e_2e_3=e_5$  & $e_2 e_4 =e_5$ & $e_3 e_2 =2e_5$  \\

$\mathcal{Z}_{37}$ &$:$& $ e_1e_2=e_4$ & $e_2 e_1 =-e_4$ & $e_2e_2=e_3$ \\
&&  $e_2e_3=e_5$ & $e_2 e_4 =e_5$ & $e_3 e_2 =2e_5$ \\


$\mathcal{Z}_{38}$ &$:$& $ e_1e_1=e_2$ & $e_1 e_2 =e_3$ & $e_1e_3=e_5$ & $e_2e_1=2e_3$ \\
&& $e_2e_2=3e_5$  & $e_3e_1=3e_5$ & $e_4e_4=e_5$ \\

$\mathcal{Z}_{39}$ &$:$& $ e_1e_1=e_2$ & $e_1 e_2 =e_3$ & $e_1e_3=e_5$ & $e_1e_4=e_5$ \\
&& $e_2e_1=2e_3$  & $e_2e_2=3e_5$ & $e_3e_1=3e_5$  \\


$\mathcal{Z}_{40}$ &$:$& $ e_1e_1=e_2$ & $e_1 e_2 =\frac{1}{2}e_3$ & $e_1e_3=2e_4$ & $e_1e_4=e_5$ & $e_2e_1=e_3$ \\
&& $e_2e_2=3e_4$ & $e_2e_3=8e_5$ & $e_3e_1=6e_4$ & $e_3e_2=12e_5$ & $e_4e_1=4e_5$ \\


$[\mathfrak{N}_1^{\mathbb{C}}]^2_{01}$ &$:$& $ e_1e_1=e_2$ &$e_1 e_2 =e_4$ &$e_1 e_3 =e_5$&$e_2 e_1 =2e_4$ \\

$[\mathfrak{N}_1^{\mathbb{C}}]^{2,\alpha}_{02}$ &$:$& $ e_1e_1=e_2$ &$e_1 e_2 =e_4$ &$e_1 e_3 =\alpha e_5$&$e_2 e_1 =2e_4$&$e_3 e_1 =e_5$\\

$[\mathfrak{N}_1^{\mathbb{C}}]^2_{03}$ &$:$& $ e_1e_1=e_2$ &$e_1 e_2 =e_4$ &$e_1 e_3 =e_5$&$e_2 e_1 =2e_4$&$e_3 e_3 =e_5$\\

$[\mathfrak{N}_1^{\mathbb{C}}]^2_{04}$ &$:$& $ e_1e_1=e_2$&$e_1 e_2 =e_4$&$e_2 e_1 =2e_4$ &$e_3 e_3 =e_5$\\

$[\mathfrak{N}_1^{\mathbb{C}}]^2_{05}$ &$:$& $ e_1e_1=e_2$ &$e_1 e_2 =e_4$ &$e_1 e_3 =e_4$&$e_2 e_1 =2e_4$ &$e_3 e_3 =e_5$ \\

$[\mathfrak{N}_1^{\mathbb{C}}]^2_{06}$ &$:$& $ e_1e_1=e_2$ &$e_1 e_2 =e_4$ &$e_1 e_3 =e_4+e_5$ &$e_2 e_1 =2e_4$ &$e_3 e_3 =e_5$\\

$[\mathfrak{N}_1^{\mathbb{C}}]^2_{07}$ &$:$& $ e_1e_1=e_2$ &$e_1 e_2 =e_4$ &$e_1 e_3 =e_5$ &$e_2 e_1 =2e_4$ & $e_3 e_1 =e_4+2e_5$\\

$[\mathfrak{N}_1^{\mathbb{C}}]^2_{08}$ &$:$& $ e_1e_1=e_2$ &$e_1 e_2 =e_4$ &$e_1 e_3 =e_5$ &$e_2 e_1 =2e_4$ &$e_3 e_3 =e_4$\\

$[\mathfrak{N}_1^{\mathbb{C}}]^{2,\alpha}_{09}$ &$:$& $ e_1e_1=e_2$ &$e_1 e_2 =e_4$
&$e_1 e_3 =\alpha e_5$\\
&&$e_2 e_1 =2e_4$ &$e_3 e_1 =e_5$  & $e_3 e_3 =e_4$ \\


$[\mathfrak{N}_1]^2_{01}$ &$:$& $e_1e_1=e_4$& $e_1e_2=  e_3$&$e_1e_3=e_5$  &$ e_2e_1= -e_3$ \\

$[\mathfrak{N}_1]^2_{02}$ &$:$& $e_1e_2=  e_3+e_4$&$e_1e_3=e_5$  &$ e_2e_1= -e_3$ \\

$[\mathfrak{N}_1]^2_{03}$ &$:$& $e_1e_2=  e_3$  &$e_1e_3=e_5$ &$ e_2e_1= -e_3$&$e_2e_2=e_4$\\

$[\mathfrak{N}_1]^2_{04}$ &$:$&$e_1e_1=e_4$& $e_1e_2=  e_3$  &$e_1e_3=e_5$&$ e_2e_1= -e_3$ &$e_2e_2=e_4$\\

$[\mathfrak{N}_1]^2_{05}$ &$:$& $e_1e_2=  e_3$ &$e_1e_3=e_5$ &$ e_2e_1= -e_3$ &$e_2e_3=e_4$ \\

$[\mathfrak{N}_1]^2_{06}$ &$:$& $e_1e_1=e_4$&$e_1e_2=  e_3$ &$e_1e_3=e_5$ &$ e_2e_1= -e_3$ &$e_2e_3=e_4$\\

$[\mathfrak{N}_1]^2_{07}$ &$:$& $e_1e_2=  e_3+e_4$ &$e_1e_3=e_5$ &$ e_2e_1= -e_3$&$e_2e_3=e_4$ \\

$[\mathfrak{N}_1]^2_{08}$ &$:$&$e_1e_1=e_4$& $e_1e_2=  e_3$ &$e_1e_3=e_5$ \\
&&$ e_2e_1= -e_3$ &$e_2e_2=e_4$ &$e_2e_3=e_4$\\

$[\mathfrak{N}_1]^2_{09}$ &$:$&$e_1e_1=e_4$& $e_1e_2=  e_3$ &$e_1e_3=e_5$ &$ e_2e_1= -e_3$&$e_2e_2=e_5$ \\

$[\mathfrak{N}_1]^2_{10}$ &$:$&  $e_1e_2=  e_3+e_4$ &$e_1e_3=e_5$ &$ e_2e_1= -e_3$&$e_2e_2=e_5$ \\
        \end{longtable}

All of these algebras are pairwise non-isomorphic, except for the following:
 $\mathcal{Z}_{02}^{\alpha} \cong  \mathcal{Z}_{02}^{\alpha^{-1}}.$ 

\end{theoremA}

\section{The geometric classification of Zinbiel  algebras}

\subsection{Definitions and notation}
Given an $n$-dimensional vector space $\mathbb V$, the set ${\rm Hom}(\mathbb V \otimes \mathbb V,\mathbb V) \cong \mathbb V^* \otimes \mathbb V^* \otimes \mathbb V$
is a vector space of dimension $n^3$. This space has the structure of the affine variety $\mathbb{C}^{n^3}$. Indeed, let us fix a basis $e_1,\dots,e_n$ of $\mathbb V$. Then any $\mu\in {\rm Hom}(\mathbb V \otimes \mathbb V,\mathbb V)$ is determined by $n^3$ structure constants $c_{ij}^k\in\mathbb{C}$ such that
$\mu(e_i\otimes e_j)=\sum\limits_{k=1}^nc_{ij}^ke_k$. A subset of ${\rm Hom}(\mathbb V \otimes \mathbb V,\mathbb V)$ is {\it Zariski-closed} if it can be defined by a set of polynomial equations in the variables $c_{ij}^k$ ($1\le i,j,k\le n$).

Let $T$ be a set of polynomial identities.
The set of algebra structures on $\mathbb V$ satisfying polynomial identities from $T$ forms a Zariski-closed subset of the variety ${\rm Hom}(\mathbb V \otimes \mathbb V,\mathbb V)$. We denote this subset by $\mathbb{L}(T)$.
The general linear group ${\rm GL}(\mathbb V)$ acts on $\mathbb{L}(T)$ by conjugations:
$$ (g * \mu )(x\otimes y) = g\mu(g^{-1}x\otimes g^{-1}y)$$
for $x,y\in \mathbb V$, $\mu\in \mathbb{L}(T)\subset {\rm Hom}(\mathbb V \otimes\mathbb V, \mathbb V)$ and $g\in {\rm GL}(\mathbb V)$.
Thus, $\mathbb{L}(T)$ is decomposed into ${\rm GL}(\mathbb V)$-orbits that correspond to the isomorphism classes of algebras.
Let $O(\mu)$ denote the orbit of $\mu\in\mathbb{L}(T)$ under the action of ${\rm GL}(\mathbb V)$ and $\overline{O(\mu)}$ denote the Zariski closure of $O(\mu)$.

Let $\mathcal A$ and $\mathcal B$ be two $n$-dimensional algebras satisfying the identities from $T$, and let $\mu,\lambda \in \mathbb{L}(T)$ represent $\mathcal A$ and $\mathcal B$, respectively.
We say that $\mathcal A$ degenerates to $\mathcal B$ and write $\mathcal A\to \mathcal B$ if $\lambda\in\overline{O(\mu)}$.
Note that in this case we have $\overline{O(\lambda)}\subset\overline{O(\mu)}$. Hence, the definition of a degeneration does not depend on the choice of $\mu$ and $\lambda$. If $\mathcal A\not\cong \mathcal B$, then the assertion $\mathcal A\to \mathcal B$ is called a {\it proper degeneration}. We write $\mathcal A\not\to \mathcal B$ if $\lambda\not\in\overline{O(\mu)}$.

Let $\mathcal A$ be represented by $\mu\in\mathbb{L}(T)$. Then  $\mathcal A$ is  {\it rigid} in $\mathbb{L}(T)$ if $O(\mu)$ is an open subset of $\mathbb{L}(T)$.
 Recall that a subset of a variety is called irreducible if it cannot be represented as a union of two non-trivial closed subsets.
 A maximal irreducible closed subset of a variety is called an {\it irreducible component}.
It is well known that any affine variety can be represented as a finite union of its irreducible components in a unique way.
The algebra $\mathcal A$ is rigid in $\mathbb{L}(T)$ if and only if $\overline{O(\mu)}$ is an irreducible component of $\mathbb{L}(T)$.

Given the spaces $U$ and $W$, we write simply $U>W$ instead of $\dim \,U>\dim \,W$.



\subsection{Method of the description of  degenerations of algebras}

In the present work we use the methods applied to Lie algebras in \cite{  GRH,GRH2,S90}.
First of all, if $\mathcal A\to \mathcal B$ and $\mathcal A\not\cong \mathcal B$, then $\mathfrak{Der}(\mathcal A)<\mathfrak{Der}(\mathcal B)$, where $\mathfrak{Der}(\mathcal A)$ is the Lie algebra of derivations of $\mathcal A$. We compute the dimensions of algebras of derivations and check the assertion $\mathcal A\to \mathcal B$ only for $\mathcal A$ and $\mathcal B$ such that $\mathfrak{Der}(\mathcal A)<\mathfrak{Der}(\mathcal B)$.


To prove degenerations, we construct families of matrices parametrized by $t$. Namely, let $\mathcal A$ and $\mathcal B$ be two algebras represented by the structures $\mu$ and $\lambda$ from $\mathbb{L}(T)$ respectively. Let $e_1,\dots, e_n$ be a basis of $\mathbb  V$ and $c_{ij}^k$ ($1\le i,j,k\le n$) be the structure constants of $\lambda$ in this basis. If there exist $a_i^j(t)\in\mathbb{C}$ ($1\le i,j\le n$, $t\in\mathbb{C}^*$) such that $E_i^t=\sum\limits_{j=1}^na_i^j(t)e_j$ ($1\le i\le n$) form a basis of $\mathbb V$ for any $t\in\mathbb{C}^*$, and the structure constants of $\mu$ in the basis $E_1^t,\dots, E_n^t$ are such rational functions $c_{ij}^k(t)\in\mathbb{C}[t]$ that $c_{ij}^k(0)=c_{ij}^k$, then $\mathcal A\to \mathcal B$.
In this case  $E_1^t,\dots, E_n^t$ is called a {\it parametrized basis} for $\mathcal A\to \mathcal B$.
To simplify our equations, we will use the notation $A_i=\langle e_i,\dots,e_n\rangle,\ i=1,\ldots,n$ and write simply $A_pA_q\subset A_r$ instead of $c_{ij}^k=0$ ($i\geq p$, $j\geq q$, $k<r$).

Since the variety of $5$-dimensional Zinbiel algebras  contains infinitely many non-isomorphic algebras, we have to do some additional work.
Let $\mathcal A(*):=\{\mathcal A(\alpha)\}_{\alpha\in I}$ be a series of algebras, and let $\mathcal B$ be another algebra. Suppose that for $\alpha\in I$, $\mathcal A(\alpha)$ is represented by the structure $\mu(\alpha)\in\mathbb{L}(T)$ and $B\in\mathbb{L}(T)$ is represented by the structure $\lambda$. Then we say that $\mathcal A(*)\to \mathcal B$ if $\lambda\in\overline{\{O(\mu(\alpha))\}_{\alpha\in I}}$, and $\mathcal A(*)\not\to \mathcal B$ if $\lambda\not\in\overline{\{O(\mu(\alpha))\}_{\alpha\in I}}$.

Let $\mathcal A(*)$, $\mathcal B$, $\mu(\alpha)$ ($\alpha\in I$) and $\lambda$ be as above. To prove $\mathcal A(*)\to \mathcal B$ it is enough to construct a family of pairs $(f(t), g(t))$ parametrized by $t\in\mathbb{C}^*$, where $f(t)\in I$ and $g(t)\in {\rm GL}(\mathbb V)$. Namely, let $e_1,\dots, e_n$ be a basis of $\mathbb V$ and $c_{ij}^k$ ($1\le i,j,k\le n$) be the structure constants of $\lambda$ in this basis. If we construct $a_i^j:\mathbb{C}^*\to \mathbb{C}$ ($1\le i,j\le n$) and $f: \mathbb{C}^* \to I$ such that $E_i^t=\sum\limits_{j=1}^na_i^j(t)e_j$ ($1\le i\le n$) form a basis of $\mathbb V$ for any  $t\in\mathbb{C}^*$, and the structure constants of $\mu_{f(t)}$ in the basis $E_1^t,\dots, E_n^t$ are such rational functions $c_{ij}^k(t)\in\mathbb{C}[t]$ that $c_{ij}^k(0)=c_{ij}^k$, then $\mathcal A(*)\to \mathcal B$. In this case  $E_1^t,\dots, E_n^t$ and $f(t)$ are called a parametrized basis and a {\it parametrized index} for $\mathcal A(*)\to \mathcal B$, respectively.

We now explain how to prove $\mathcal A(*)\not\to\mathcal  B$.
Note that if $\mathfrak{Der} \ \mathcal A(\alpha)  > \mathfrak{Der} \  \mathcal B$ for all $\alpha\in I$ then $\mathcal A(*)\not\to\mathcal B$.
One can also use the following  Lemma, whose proof is the same as the proof of Lemma 1.5 from \cite{GRH}.

\begin{lemma}
Let $\mathfrak{B}$ be a Borel subgroup of ${\rm GL}(\mathbb V)$ and $\mathcal{R}\subset \mathbb{L}(T)$ be a $\mathfrak{B}$-stable closed subset.
If $\mathcal A(*) \to \mathcal B$ and for any $\alpha\in I$ the algebra $\mathcal A(\alpha)$ can be represented by a structure $\mu(\alpha)\in\mathcal{R}$, then there is $\lambda\in \mathcal{R}$ representing $\mathcal B$.
\end{lemma}

\subsection{The geometric classification of $5$-dimensional  
 Zinbiel algebras}
The main result of the present section is the following theorem.

\begin{theoremB}
The variety of complex  $5$-dimensional Zinbiel algebras  has 
dimension  $24$   and it has 
$16$  irreducible components
defined by  
 
\begin{center}
$\mathcal{C}_1=\overline{\{\mathcal{O}({\mathfrak V}_{4+1})\}},$ \, $\mathcal{C}_2=\overline{\{\mathcal{O}({\mathfrak V}_{3+2})\}},$ \,
$\mathcal{C}_3=\overline{\mathcal{O}([\mathfrak{N}_1]^{2}_{08})},$ \,
$\mathcal{C}_4=\overline{\mathcal{O}([\mathfrak{N}_1^{\mathbb{C}}]^{2}_{06})},$ \,
$\mathcal{C}_5=\overline{\{\mathcal{O}(\mathcal{Z}_{02}^\alpha)\}},$ \,
$\mathcal{C}_6=\overline{\mathcal{O}(\mathcal{Z}_{05})},$ \,
$\mathcal{C}_7=\overline{\{\mathcal{O}(\mathcal{Z}_{14}^\alpha)\}},$ \,
$\mathcal{C}_8=\overline{\mathcal{O}(\mathcal{Z}_{22})},$ \,
$\mathcal{C}_{9}=\overline{\mathcal{O}(\mathcal{Z}_{23})},$ \,
$\mathcal{C}_{10}=\overline{\mathcal{O}(\mathcal{Z}_{24})},$ \,
$\mathcal{C}_{11}=\overline{\mathcal{O}(\mathcal{Z}_{27})},$ \,
$\mathcal{C}_{12}=\overline{\{\mathcal{O}(\mathcal{Z}_{30}^{\alpha})\}},$ \,
$\mathcal{C}_{13}=\overline{\mathcal{O}(\mathcal{Z}_{34})},$ \,
$\mathcal{C}_{14}=\overline{\mathcal{O}(\mathcal{Z}_{35})},$ \,
$\mathcal{C}_{15}=\overline{\mathcal{O}(\mathcal{Z}_{38})},$ \,
$\mathcal{C}_{16}=\overline{\mathcal{O}(\mathcal{Z}_{40})}.$
\end{center}
In particular, the variety of  complex  $5$-dimensional Zinbiel algebras  has  $11$ rigid algebras
\begin{center}$[\mathfrak{N}_1]^{2}_{08}, \, 
[\mathfrak{N}_1^{\mathbb{C}}]^{2}_{06}, \, 
\mathcal{Z}_{05}, \,
\mathcal{Z}_{22}, \,
\mathcal{Z}_{23}, \,
\mathcal{Z}_{24}, \,
\mathcal{Z}_{27}, \,
\mathcal{Z}_{34}, \,
\mathcal{Z}_{35}, \,
\mathcal{Z}_{38} \mbox{ and }
\mathcal{Z}_{40}.$
\end{center}

\end{theoremB}

\begin{proof}
Thanks to \cite{ikp20}, the variety of $5$-dimensional $2$-step nilpotent algebras has only three irreducible components defined by 

\begin{longtable}{lllllll}
${\mathfrak V}_{4+1}$ & $:$&  
$e_1e_2=e_5$& $e_2e_1=\lambda e_5$ &$e_3e_4=e_5$&$e_4e_3=\mu e_5$\\

${\mathfrak V}_{3+2}$ &$ :$&
$e_1e_1 =  e_4$& $e_1e_2 = \mu_1 e_5$ & $e_1e_3 =\mu_2 e_5$& 
$e_2e_1 = \mu_3 e_5$  & $e_2e_2 = \mu_4 e_5$  \\
& & $e_2e_3 = \mu_5 e_5$  & $e_3e_1 = \mu_6 e_5$  & \multicolumn{2}{l}{$e_3e_2 = \lambda e_4+ \mu_7 e_5$ } & $e_3e_3 =  e_5$  \\

${\mathfrak V}_{2+3}$ &$ :$&
$e_1e_1 = e_3 + \lambda e_5$& $e_1e_2 = e_3$ & $e_2e_1 = e_4$& $e_2e_2 = e_5$

\end{longtable}

Thanks to \cite{kppv,ikp20}, 
all $5$-dimensional split Zinbiel  algebras are in the orbit closure 
of the families ${\mathfrak V}_{4+1}$ and ${\mathfrak V}_{3+2},$
and the algebras
$[\mathfrak{Z}_1]^1_1, \, [\mathfrak{N}_1]^1_{01}, \, [\mathfrak{N}_1^{\mathbb{C}}]^1_{01}.$  

Let us give some useful degenerations for our proof.
  \begin{longtable}{|lcll|}
 
\hline

{$\mathcal{Z}_{04}$}&$\to$&$\mathcal{Z}_{01}$ & 
$E_1^t=e_1+\frac{2(3t^2-1)}{9t^2}e_4$ \\ 
\multicolumn{3}{|l}{$E_2^t=e_2+\frac{2(3t^2-1)}{9t^2}e_5$} & 
$E_3^t=-\frac{1}{3}e_2+\frac{3t^2}{3t^2-1}e_3$ \\ \multicolumn{3}{|l}{$E_4^t=-\frac{1}{9t}e_2+\frac{t}{3t^2-1}e_3+\frac{3t^2-1}{3t}e_4$} & $E_5^t=e_5$\\ \hline

$\mathcal{Z}_{02}^{t-1}$&$\to$&$\mathcal{Z}_{03}$ & 
$E_1^t=e_1$ \\ 
\multicolumn{3}{|l}{$E_2^t=e_2$} & 
$E_3^t=e_3+\frac{1}{t}e_4$\\  
\multicolumn{3}{|l}{$E_4^t=e_4$} & $E_5^t=e_5$\\ \hline

$\mathcal{Z}_{02}^{t+2}$&$\to$&$\mathcal{Z}_{04}$ & 
$E_1^t=e_1-\frac{2}{t}e_3$ \\ 
\multicolumn{3}{|l}{$E_2^t=e_2$} & 
$E_3^t=e_3$\\   
\multicolumn{3}{|l}{$E_4^t=\frac{t+2}{t}e_2+e_4$} & $E_5^t=e_5$\\ \hline

{$\mathcal{Z}_{22}$} &$\to$&$\mathcal{Z}_{07}$ & 
$E_1^t=e_1+\frac{\sqrt{4t^2-5}}{2(t^2-1)}e_3+\frac{2}{\sqrt{4t^2-5}}e_4$ \\ \multicolumn{3}{|l}{$E_2^t=\frac{t}{2(1-t^2)}e_1+\frac{t\sqrt{4t^2-5}}{2(t^2-1)}e_2$}&
$E_3^t=\frac{t\sqrt{4t^2-5}}{2(t^2-1)}e_3-\frac{t}{2(t^2-1)}e_5$\\ 
\multicolumn{3}{|l}{$E_4^t=\frac{2t}{\sqrt{4t^2-5}}e_4$} & $E_5^t=\frac{t}{t^2-1}e_5$\\ \hline

{$\mathcal{Z}_{22}$} &$\to$&$\mathcal{Z}_{08}$ & 
$E_1^t=e_1+\frac{\sqrt{t(4t^2-t-4)}}{2t(t^2-1)}e_3+2\sqrt{\frac{t}{4t^2-t-4}}e_4$ \\ \multicolumn{3}{|l}{$E_2^t=\frac{t}{2(1-t^2)}e_1+\frac{\sqrt{t(4t^2-t-4)}}{2(t^2-1)}e_2$}&
$E_3^t=\frac{\sqrt{t(4t^2-t-4)}}{2(t^2-1)}e_3+\frac{t}{2(1-t^2)}e_5$\\ \multicolumn{3}{|l}{$E_4^t=2t\sqrt{\frac{t}{4t^2-t-4}}e_4$} & $E_5^t=\frac{t}{t^2-1}e_5$\\ \hline

$\mathcal{Z}_{10}^{-\frac{1}{t}}$&$\to$&$\mathcal{Z}_{09}$ & 
$E_1^t=e_1+e_2$ \\
\multicolumn{3}{|l}{$E_2^t=te_2$} & $E_3^t=te_3$\\ 
\multicolumn{3}{|l}{$E_4^t=e_3+te_4$} & $E_5^t=te_5$\\ \hline

{$\mathcal{Z}_{14}^{\frac{t^2-\alpha}{(\alpha-1)^2}}$}&$\to$&$\mathcal{Z}_{10}^\alpha$ & 
$E_1^t=\frac{(\alpha-1)^2}{t}e_1+\frac{\alpha-1}{t}e_4$ \\ 
\multicolumn{3}{|l}{$E_2^t=e_2$} &
$E_3^t=\frac{(\alpha-1)^2}{t}e_3$\\
\multicolumn{3}{|l}{$E_4^t=(\alpha-1)e_4$} & $E_5^t=\frac{(\alpha-1)^2}{t}e_5$\\ \hline

$\mathcal{Z}_{14}^{\frac{t+1}{4}}$&$\to$&$\mathcal{Z}_{11}$ & 
$E_1^t=\frac{4}{t}e_1-\frac{2}{t}e_4$ \\ 
\multicolumn{3}{|l}{$E_2^t=e_2$} &
$E_3^t=\frac{4}{t}e_3$\\ 
\multicolumn{3}{|l}{$E_4^t=-2e_4$} & $E_5^t=\frac{4}{t}e_5$\\ \hline

{$\mathcal{Z}_{22}$}&$\to$&$\mathcal{Z}_{17}$ & 
$E_1^t=\frac{1}{t}e_1$ \\ 
\multicolumn{3}{|l}{$E_2^t=\frac{1}{2t^2}e_1+\frac{\sqrt{4t^2-1}}{2t^2}e_2$} & $E_3^t=\frac{\sqrt{4t^2-1}}{2t^3}e_3-\frac{1}{2t^3}e_5$\\
\multicolumn{3}{|l}{$E_4^t=-2\sqrt{4t^2-1}e_4$} & $E_5^t=\frac{1}{t^3}e_5$\\ \hline

$\mathcal{Z}_{22}$&$\to$&$\mathcal{Z}_{20}$ & 
$E_1^t=e_1$ \\ 
\multicolumn{3}{|l}{$E_2^t=\frac{1}{2t}e_1+\frac{\sqrt{4t^2-1}}{2t}e_2$} &
$E_3^t=\frac{\sqrt{4t^2-1}}{2t}e_3-\frac{1}{2t}e_5$ \\ 
\multicolumn{3}{|l}{$E_4^t=\frac{2}{\sqrt{4t^2-1}}e_4$} & $E_5^t=\frac{1}{t}e_5$\\ \hline

{$\mathcal{Z}_{24}$}&$\to$&$\mathcal{Z}_{25}$ & 
$E_1^t=-(t+1)e_2-\frac{(t+1)(t+2)}{t}e_3-\frac{(t+1)}{t}e_4$\\
\multicolumn{3}{|l}{$E_2^t=-te_1+e_2$} &
$E_3^t=-t^2(t+1)e_3-(t+1)e_4+\frac{(t+1)^2}{t}e_5$\\ 
\multicolumn{3}{|l}{$E_4^t=-t(t+1)e_3+(t+1)e_4$} & $E_5^t=-(t+1)^2e_5$  \\\hline

$\mathcal{Z}_{27}$&$\to$&$\mathcal{Z}_{28}$ & 
$E_1^t=e_2$ \\
\multicolumn{3}{|l}{$E_2^t=te_1+e_3$} & $E_3^t=te_4-e_5$\\ 
\multicolumn{3}{|l}{$E_4^t=te_3$} & $E_5^t=-te_5$\\ \hline

{$\mathcal{Z}_{30}^{t}$}&$\to$&$\mathcal{Z}_{31}$ & 
$E_1^t=-\frac{2t^2}{t+1}e_1+e_4$ \\ 
\multicolumn{3}{|l}{$E_2^t=e_2$} & $E_3^t=e_3$\\ 
\multicolumn{3}{|l}{$E_4^t=-\frac{2t^2}{t+1}e_4+2(t+1)e_5$} & $E_5^t=-2t^2e_5$\\ \hline

$\mathcal{Z}_{35}$&$\to$&$\mathcal{Z}_{36}$ & 
$E_1^t=\sqrt{t}e_1-\frac{1}{3t}e_3+\frac{1}{\sqrt{t}}e_4$ \\ 
\multicolumn{3}{|l}{$E_2^t=\frac{1}{\sqrt{t}}e_1+e_2$} & 
$E_3^t=e_3$ \\
\multicolumn{3}{|l}{$E_4^t=\sqrt{t}e_4-\frac{2}{3t}e_5$} & $E_5^t=e_5$\\ \hline

$\mathcal{Z}_{35}$&$\to$&$\mathcal{Z}_{37}$ & 
$E_1^t=\sqrt{t}e_1$ \\ 
\multicolumn{3}{|l}{$E_2^t=\frac{1}{\sqrt{t}}e_1+e_2$} & $E_3^t=e_3$\\ 
\multicolumn{3}{|l}{$E_4^t=\sqrt{t}e_4$} & $E_5^t=e_5$\\ \hline

$\mathcal{Z}_{38}$&$\to$&$\mathcal{Z}_{39}$ & 
$E_1^t=e_1+\frac{3}{2\sqrt{t}}e_4$ \\ 
\multicolumn{3}{|l}{$E_2^t=e_2+\frac{9}{4t}e_5$} & 
$E_3^t=e_3$\\ 
\multicolumn{3}{|l}{$E_4^t=-\frac{1}{2}e_3+\sqrt{t}e_4$} & $E_5^t=e_5$\\ \hline

{$[\mathfrak{N}_1^{\mathbb{C}}]^2_{06}$} & $\to$ & $[\mathfrak{N}_1^{\mathbb{C}}]^{2,\alpha}_{02}$ & 
$E_1^t=\frac{1}{\alpha-1}e_1+\frac{1}{(\alpha-1)^2}e_3$ \\ \multicolumn{3}{|l}{$E_2^t=\frac{1}{(\alpha-1)^2}e_2+\frac{1}{(\alpha-1)^3}e_4+\frac{\alpha}{(\alpha-1)^4}e_5$} & $E_3^t=\frac{t}{(\alpha-1)^2}e_3$ \\
\multicolumn{3}{|l}{$E_4^t=\frac{1}{(\alpha-1)^3}e_4$} & $E_5^t=\frac{t}{(\alpha-1)^4}e_5$\\ \hline

$[\mathfrak{N}_1^{\mathbb{C}}]^{2,\frac{1}{2}}_{09}$&$\to$&$[\mathfrak{N}_1^{\mathbb{C}}]^2_{07}$ & 
$E_1^t=te_1-te_3$ \\ 
\multicolumn{3}{|l}{$E_2^t=t^2e_2+t^2e_4-\frac{3t^2}{2}e_5$} &
$E_3^t=t^2e_2+t^2e_3$ \\ 
\multicolumn{3}{|l}{$E_4^t=t^3e_4$} & $E_5^t=\frac{t^{3}}{2}e_5$\\ \hline

{$[\mathfrak{N}_1^{\mathbb{C}}]^2_{06}$} &$\to$& $[\mathfrak{N}_1^{\mathbb{C}}]^{2,\alpha}_{09}$ & $E_1^t=-\frac{2t^2}{(\alpha-1)(2\alpha-2t^2-1)}\left(e_1+\frac{1}{\alpha-1}e_3\right)$   \\ 

\multicolumn{4}{|l|}{$E_2^t=\frac{4t^4}{(\alpha-1)^2(2\alpha-2t^2-1)^{2}}\left(e_2+\frac{1}{\alpha-1}e_4+\frac{\alpha}{(\alpha-1)^2}e_5\right)$} \\ 
\multicolumn{4}{|l|}{$E_3^t=-\frac{2t^3}{(\alpha-1)^2(2\alpha-2t^2-1)}\left(\frac{1}{2\alpha-2t^2-1}e_2+e_3\right)$} \\
\multicolumn{3}{|l}{$E_4^t=-\frac{8t^6}{(\alpha-1)^3(2\alpha-2t^2-1)^{3}}e_4$}  &
 $E_5^t=\frac{4t^5}{(\alpha-1)^3(2\alpha-2t^2-1)^{2}}\left(\frac{2}{2\alpha-2t^2-1} e_4+\frac{1}{\alpha-1}e_5\right)$ \\ \hline

$[\mathfrak{N}_1]^2_{02}$&$\to$&$[\mathfrak{N}_1]^2_{01}$ & $E_1^t=e_1+e_2$ \\
\multicolumn{3}{|l}{$E_2^t=te_2$} &  $E_3^t=te_3$\\ 
\multicolumn{3}{|l}{$E_4^t=e_4$} & $E_5^t=te_5$\\ \hline

$[\mathfrak{N}_1]^2_{07}$&$\to$&$[\mathfrak{N}_1]^2_{06}$ & 
$E_1^t=e_1+\frac{1}{t^2}e_2$ \\ 
\multicolumn{3}{|l}{$E_2^t=\frac{1}{t}e_2$} & 
$E_3^t=\frac{1}{t}e_3$\\ 
\multicolumn{3}{|l}{$E_4^t=\frac{1}{t^2}e_4$} & $E_5^t=\frac{1}{t^3}e_4+\frac{1}{t}e_5$\\ \hline

{$[\mathfrak{N}_1]^2_{08}$}&$\to$&$[\mathfrak{N}_1]^2_{07}$ & 
$E_1^t=\frac{1}{t}e_1+\frac{1}{t\sqrt{4t^2-1}}e_2$ \\ 
\multicolumn{3}{|l}{$E_2^t=\frac{2}{\sqrt{4t^2-1}}e_2$}
 & $E_3^t=\frac{2}{t\sqrt{4t^2-1}}e_3-\frac{2}{t(4t^2-1)}e_4$ \\ 
 \multicolumn{3}{|l}{$E_4^t=\frac{4}{t(4t^2-1)}e_4$} & $E_5^t=\frac{2}{t^2(4t^2-1)}e_4+\frac{2}{t^2\sqrt{4t^2-1}}e_5$\\ \hline

$[\mathfrak{N}_1]^2_{10}$&$\to$&$[\mathfrak{N}_1]^2_{09}$ & 
$E_1^t=e_1+\frac{1}{t}e_2$\\ 
\multicolumn{3}{|l}{$E_2^t=e_2-\frac{1}{t}e_3$} & 
$E_3^t=e_3-\frac{1}{t}e_5$\\  
\multicolumn{3}{|l}{$E_4^t=\frac{1}{t}e_4+\frac{1}{t^2}e_5$} & $E_5^t=e_5$\\ \hline

{$[\mathfrak{N}_1]^2_{08}$}&$\to$&$[\mathfrak{N}_1]^2_{10}$ & 
$E_1^t=te_2-te_3$ \\ 
\multicolumn{3}{|l}{$E_2^t=-t^2e_1$} &
$E_3^t=t^3e_3-t^3e_5$\\ 
\multicolumn{3}{|l}{$E_4^t=t^3e_5$} & $E_5^t=t^4e_4$\\ \hline

{$[\mathfrak{N}_1]^2_{08}$}&$\to$&$[\mathfrak{N}_1]^1_{01}$ & 
$E_1^t=te_1$ \\
\multicolumn{3}{|l}{$E_2^t=\sqrt{t}e_2$} &
$E_3^t=\sqrt{t^3}e_3$\\  
\multicolumn{3}{|l}{$E_4^t=te_4$} & $E_5^t=-e_4+\sqrt{t^3}e_5$\\ \hline

{$[\mathfrak{N}_1^\mathbb{C}]^2_{06}$}&$\to$&$[\mathfrak{N}_1^\mathbb{C}]^1_{01}$ & 
$E_1^t=\frac{1}{t-1}e_1$ \\ 
\multicolumn{3}{|l}{$E_2^t=\frac{1}{(t-1)^2}e_2$} &$E_3^t=\frac{1}{t-1}e_3$\\ \multicolumn{3}{|l}{$E_4^t=\frac{1}{(t-1)^3}e_4$} & $E_5^t=-\frac{1}{t(t-1)^3}e_4+\frac{1}{t(t-1)^2}e_5$\\ \hline

{$\mathcal{Z}_{40}$}&$\to$&$\mathfrak{V}_{2+3}$ & 
$E_1^t=e_1-\frac{1}{t}e_2+\frac{3(2+9\lambda)}{8t^2}e_3-\frac{6(2+9\lambda)}{5t^3}e_4$ \\ 
\multicolumn{3}{|l}{$E_2^t=-\frac{3}{t}e_2$} &$E_3^t=-\frac{3}{2t}e_3+\frac{9}{t^2}e_4-\frac{27(2+9\lambda)}{2t^3}e_5$\\ \multicolumn{3}{|l}{$E_4^t=-\frac{3}{t}e_3+\frac{9}{t^2}e_4-\frac{9(2+9\lambda)}{t^3}e_5$} & $E_5^t=\frac{27}{t^2}e_4$\\ \hline
\end{longtable}


For the rest of degenerations, in  case of  $E_1^t,\dots, E_n^t$ is a {\it parametric basis} for ${\bf A}\to {\bf B},$ it will be denote as
${\bf A}\xrightarrow{(E_1^t,\dots, E_n^t)} {\bf B}$.

\begin{longtable}{|lcl|lcl|}
  
\hline

$\mathcal{Z}_{09}$ & $  \xrightarrow{ (te_1,e_2,te_3,t^2e_4,t^2e_5)}$ & $\mathcal{Z}_{06}$ &

$\mathcal{Z}_{13}$ & $  \xrightarrow{ (te_1,e_2,te_3,t^{\frac{1}{2}}e_4,te_5)}$ & $\mathcal{Z}_{12}$ \\ \hline 

$\mathcal{Z}_{14}^{\frac{1}{t^2}}$ & $  \xrightarrow{ (t^2e_1,e_2,t^2e_3,te_4,t^2e_5)}$ & $\mathcal{Z}_{13}$ &

$\mathcal{Z}_{16}$ & $  \xrightarrow{ (te_1,e_2,te_3,e_4,te_5)}$ & $\mathcal{Z}_{15}$ \\ \hline 

{$\mathcal{Z}_{22}$} & $  \xrightarrow{ (t^{-1}e_1,t^{-\frac{1}{2}}e_2,t^{-\frac{3}{2}}e_3,t^{-\frac{1}{2}}e_4,t^{-2}e_5)}$ & $\mathcal{Z}_{16}$ & 

{$\mathcal{Z}_{22}$} & $  \xrightarrow{ (e_1,e_2,e_3,t^{-1}e_4,t^{-1}e_5)}$ & $\mathcal{Z}_{18}$ \\ \hline

{$\mathcal{Z}_{22}$} & $  \xrightarrow{ (e_1,te_2,te_3,t^{-1}e_4,e_5)}$ & $\mathcal{Z}_{19}$ &

{$\mathcal{Z}_{22}$} & $  \xrightarrow{ (e_1,t^{-\frac{1}{2}}e_2,t^{-\frac{1}{2}}e_3,t^{-\frac{1}{2}}e_4,t^{-1}e_5)}$ & $\mathcal{Z}_{21}$ \\ \hline

$\mathcal{Z}_{25}$ & $  \xrightarrow{ (t^{-1}e_1,t^{-1}e_2,t^{-2}e_3,t^{-2}e_4,t^{-3}e_5)}$ & $\mathcal{Z}_{26}$ &

$\mathcal{Z}_{28}$ & $  \xrightarrow{ (e_1,te_2,te_3,te_4,te_5)}$ & $\mathcal{Z}_{29}$ \\ \hline
 
$\mathcal{Z}_{34}$ & $  \xrightarrow{ (t^{-2}e_1,t^{-1}e_2,t^{-2}e_3,t^{-3}e_4,t^{-4}e_5)}$ & $\mathcal{Z}_{32}$ &

$\mathcal{Z}_{34}$ & $  \xrightarrow{ (t^{-1}e_1,t^{-1}e_2,t^{-2}e_3,t^{-2}e_4,t^{-3}e_5)}$ & $\mathcal{Z}_{33}$ \\ \hline

$[\mathfrak{N}_1^{\mathbb{C}}]^{2,\frac{1}{t}}_{02}$ & $  \xrightarrow{ (te_1,t^2e_2,e_3,t^3e_4,e_5)}$ & $[\mathfrak{N}_1^{\mathbb{C}}]^2_{01}$ &

$[\mathfrak{N}_1^{\mathbb{C}}]^2_{06}$ & $  \xrightarrow{ (t^{-1}e_1,t^{-2}e_2,t^{-1}e_3,t^{-3}e_4,t^{-2}e_5)} $ &   
$[\mathfrak{N}_1^{\mathbb{C}}]^2_{03}$\\ \hline

$[\mathfrak{N}_1^{\mathbb{C}}]^2_{05}$ & $  \xrightarrow{ (e_1,e_2,te_3,e_4,t^2e_5)} $ &   $[\mathfrak{N}_1^{\mathbb{C}}]^2_{04}$ &

$[\mathfrak{N}_1^{\mathbb{C}}]^2_{06}$ & $  \xrightarrow{ (t^{-1}e_1,t^{-2}e_2,t^{-2}e_3,t^{-3}e_4,t^{-4}e_5)} $ &   $[\mathfrak{N}_1^{\mathbb{C}}]^2_{05}$ \\ \hline

$[\mathfrak{N}_1^{\mathbb{C}}]^{2,\frac{1}{t}}_{09}$ & $  \xrightarrow{ (t^{\frac{2}{3}}e_1,t^{\frac{4}{3}}e_2,te_3,t^2e_4,t^{\frac{2}{3}}e_5)} $ &   $[\mathfrak{N}_1^{\mathbb{C}}]^2_{08}$ &

$[\mathfrak{N}_1]^2_{07}$ & $  \xrightarrow{ (e_1,te_2,te_3,te_4,te_5)} $ &   $[\mathfrak{N}_1]^2_{02}$ \\ \hline

$[\mathfrak{N}_1]^2_{04}$ & $  \xrightarrow{ (t^{\frac{1}{2}}e_1,e_2,t^{\frac{1}{2}}e_3,e_4,te_5)} $ &   $[\mathfrak{N}_1]^2_{03}$ &

$[\mathfrak{N}_1]^2_{08}$ & $  \xrightarrow{ (te_1,te_2,t^2e_3,t^2e_4,t^3e_5)} $ &   $[\mathfrak{N}_1]^2_{04}$ \\ \hline

$[\mathfrak{N}_1]^2_{06}$ & $  \xrightarrow{ (te_1,e_2,te_3,te_4,t^2e_5)} $ &   $[\mathfrak{N}_1]^2_{05}$ &

{$\mathcal{Z}_{40}$} & $  \xrightarrow{ (e_1,e_2,e_3,e_4,t^{-1}e_5)} $ &   $[\mathfrak{Z}_1]^1_{1}$ \\ \hline

\end{longtable} 

The following dimensions of orbit closures are important for us: 

\begin{longtable}{llll}  
$\dim  \mathcal{O}({\mathfrak V}_{3+2})=24,$ & 
$\dim \mathcal{O}(\mathcal{Z}_{22})=22,$ &
$\dim \mathcal{O}(\mathcal{Z}_{14}^\alpha)=21,$ &
$\dim \mathcal{O}([\mathfrak{N}_1]^{2}_{08})=21,$ \\
$\dim \mathcal{O}({\mathfrak V}_{4+1})=20,$ &
$\dim \mathcal{O}(\mathcal{Z}_{02}^\alpha)=20,$ &
$\dim \mathcal{O}(\mathcal{Z}_{30}^{\alpha})=20,$ &
$\dim \mathcal{O}(\mathcal{Z}_{05})=20,$ \\
$\dim \mathcal{O}(\mathcal{Z}_{23})=20,$ &
$\dim \mathcal{O}(\mathcal{Z}_{24})=20,$ &
$\dim \mathcal{O}(\mathcal{Z}_{27})=20,$ &
$\dim \mathcal{O}(\mathcal{Z}_{35})=20,$ \\
$\dim \mathcal{O}(\mathcal{Z}_{38})=20,$ &
$\dim \mathcal{O}(\mathcal{Z}_{40})=20,$ &
$\dim \mathcal{O}([\mathfrak{N}_1^{\mathbb{C}}]^{2}_{06})=20,$ &
$\dim \mathcal{O}(\mathcal{Z}_{34})=19.$ \,

 \end{longtable}
Let us, also, give the list of dimensions of the square of principal algebras:

\begin{longtable}{|l|l|}
\hline
 {\rm dim} $A^2=2$ &  $  \mathcal{Z}_{14}^\alpha, \, \mathcal{Z}_{02}^\alpha, \, \mathcal{Z}_{22}, {\mathfrak V}_{3+2}$\\ 
\hline {\rm dim} $A^2=3$ & $[\mathfrak{N}_1]^{2}_{08}, \,
[\mathfrak{N}^{\mathbb{C}}_1]^{2}_{06}, \,
 \mathcal{Z}_{05}, \, \mathcal{Z}_{23}, \, \mathcal{Z}_{24}, \,  \mathcal{Z}_{27}, \,    \mathcal{Z}_{30}^{\alpha}, \,  \mathcal{Z}_{34}, \, \mathcal{Z}_{35}, \, \mathcal{Z}_{38}$ \\
\hline {\rm dim} $A^2=4$  & $\mathcal{Z}_{40}$\\
\hline
\end{longtable}

Hence,  and taking into account the annihilator dimension of this algebras (see item $(2)$ in the beginning of Section \ref{secteoA}),  
$[\mathfrak{N}_1]^{2}_{08}, \, \mathcal{Z}_{05}, \, \mathcal{Z}_{22}, \, \mathcal{Z}_{23}, \, \mathcal{Z}_{24}, \,  \mathcal{Z}_{27}, \,    \mathcal{Z}_{30}^{\alpha}, \,  \mathcal{Z}_{35}, \, \mathcal{Z}_{38}, \, \mathcal{Z}_{40}$ and ${\mathfrak V}_{3+2}$ give $11$ irreducible components.
Below we have listed all necessary reasons for   non-degenerations, 
which imply that 
$[\mathfrak{N}^{\mathbb{C}}_1]^{2}_{06},
\mathcal{Z}_{02}^{\alpha}, \mathcal{Z}_{14}^{\alpha}, \mathcal{Z}_{34}$ and ${\mathfrak V}_{4+1}$ give another $5$ irreducible components.

\begin{longtable}{|rcl|l|}
\hline
\multicolumn{4}{|c|}{\textrm{
{\bf  Non-degenerations reasons}}
}  \\
\hline

$[\mathfrak{N}_1]^{2}_{08}$  &$\not \to$& 
$  \begin{array}{l}
[\mathfrak{N}^{\mathbb{C}}_1]^{2}_{06}
\end{array} $
& ${\mathcal R}=
\left\{
\begin{array}{l}
 A_1^2 \subseteq A_3, \, A_1A_4+A_4A_1=0, \\ 
 c_{11}^3=0, \, c_{22}^3=0, \, c_{12}^3=-c_{21}^3

\end{array}
\right\}$
 \\
\hline



$\mathcal{Z}_{05}$  &$\not \to$& 
$  \begin{array}{l}

\mathcal{Z}_{34}
\end{array} $
& ${\mathcal R}=
\left\{
\begin{array}{l}
 
 A_1^2 \subseteq A_3, \, A_1A_2+A_2A_1 \subseteq A_4, \, A_4A_3+A_3A_4=0,\\
 c_{12}^4=c_{21}^4, \, c_{12}^3=c_{21}^3=c_{22}^3=0,\, 
2c_{24}^5 = c_{42}^5, \, 2c_{14}^5 =  c_{41}^5

\end{array}
\right\}$
 \\
\hline

$\mathcal{Z}_{14}^{\alpha}$  &$\not \to$& 
$  \begin{array}{l}
\mathcal{Z}_{02}^{\alpha}, \\
{\mathfrak V}_{4+1}
\end{array} $
& ${\mathcal R}=
\left\{
\begin{array}{l}
A_1^2 \subseteq A_4, \, A_4A_1=0\\

\mbox{new base for }\mathcal{Z}_{14}^{\alpha}:
f_1=e_1, f_2=e_2, f_3=e_4, f_4=e_3, f_5=e_5 \\

\end{array}
\right\}$
 \\
\hline

$\mathcal{Z}_{22}$  &$\not \to$& 
$  \begin{array}{l}
\mathcal{Z}_{02}^{\alpha}, \\
\mathcal{Z}_{14}^{\alpha}, \\ 
{\mathfrak V}_{4+1}
\end{array} $
& ${\mathcal R}=
\left\{
\begin{array}{l}

A_1^2 \subseteq A_4, \,
A_1A_3+A_3A_1+A_2^2 \subseteq A_5,\,
A_4A_1+A_1A_5=0, \\
\mbox{new base for }\mathcal{Z}_{22}:
f_1=e_1, f_2=e_2, f_3=e_4, f_4=e_3, f_5=e_5 \\ 
\end{array}
\right\}$
 \\
\hline

$\mathcal{Z}_{23}$  &$\not \to$& 
$  \begin{array}{l}

\mathcal{Z}_{34}
\end{array} $
& ${\mathcal R}=
\left\{
\begin{array}{l}
 A_1^2 \subseteq A_3, \, 
 A_1A_3+A_3A_1 \subseteq A_5, \,
 A_1A_5+A_4A_1=0

\end{array}
\right\}$
 \\
\hline

$\mathcal{Z}_{24}$  &$\not \to$& 
$  \begin{array}{l}
    
\mathcal{Z}_{34}
\end{array} $
& ${\mathcal R}=
\left\{

\begin{array}{l}
A_1^2 \subseteq A_3, \, 
A_1A_2 \subseteq A_4, \\ 
A_1A_3+A_3A_1 \subseteq A_5, \,
A_1A_5+A_5A_1=0,\\

c_{22}^4 c_{11}^3=  c_{12}^4 c_{21}^3, \, 
c_{23}^5 c_{11}^3= c_{13}^5 c_{21}^3, \,
2 c_{22}^4 c_{14}^5=  c_{12}^4 c_{42}^5, \\
c_{21}^4 c_{14}^5 +  c_{12}^4 c_{14}^5 + c_{13}^5 c_{21}^3 =     c_{12}^4 c_{41}^5,\\
c_{11}^4 c_{14}^5 c_{21}^3 +  c_{11}^3 (c_{12}^4 c_{14}^5 + c_{13}^5 c_{21}^3) =c_{11}^3 c_{12}^4 c_{41}^5

\end{array}
\right\}$
 \\
\hline

$\mathcal{Z}_{27}$  &$\not \to$& 
$  \begin{array}{l}

\mathcal{Z}_{34}
\end{array} $
& ${\mathcal R}=
\left\{
A_1^2 \subseteq A_3, \,  A_3A_1=0
\begin{array}{l}

\end{array}
\right\}$
 \\

\hline

$\mathcal{Z}_{30}^{\alpha}$  &$\not \to$& 
$  \begin{array}{l}

\mathcal{Z}_{34}
\end{array} $
& ${\mathcal R}=
\left\{
\begin{array}{l}
 A_1^2 \subseteq A_3, \, A_1A_3+A_3A_1 \subseteq A_5, \, A_1A_5+A_5A_1=0,\\

c_{12}^3 = c_{21}^3, \, 
c_{23}^5 c_{42}^5=  c_{24}^5 c_{32}^5, \\


c_{31}^5 c_{24}^5 (c_{12}^4 (c_{14}^5)^2 + 3 c_{11}^3 c_{14}^5 c_{23}^5 -  2 c_{11}^3 c_{13}^5 c_{24}^5) =\\ 
\multicolumn{1}{r}{ 2c_{14}^5 (c_{12}^4 c_{14}^5 + 2 c_{11}^3 c_{23}^5) (c_{14}^5 c_{23}^5 - c_{13}^5 c_{24}^5)}  


\end{array}
\right\}$
 \\
\hline

$\mathcal{Z}_{35}$  &$\not \to$& 
$  \begin{array}{l}

\mathcal{Z}_{34}
\end{array} $
& ${\mathcal R}=
\left\{
\begin{array}{l}
 A_1^2 \subseteq A_3, \, A_4A_1+A_3^2+A_1A_5=0, \\

c_{11}^3 (c_{12}^4  +  c_{21}^4)=2 c_{11}^4 c_{12}^3,\,
c_{22}^3(c_{12}^4  + c_{21}^4)= 2 c_{21}^3 c_{22}^4

\end{array}
\right\}$
 \\
\hline

$\mathcal{Z}_{38}$  &$\not \to$& 
$  \begin{array}{l}

\mathcal{Z}_{34}
\end{array} $
& ${\mathcal R}=
\left\{
\begin{array}{l}
 A_1^2 \subseteq A_3, \, A_1A_2+A_2A_1 \subseteq A_4,\\ A_2^2+A_1A_4+A_4A_1\subseteq A_5, \, A_1A_5+A_5A_1=0,\\
  c_{14}^5 c_{11}^4 +c_{11}^3   c_{31}^5 = 2c_{11}^3  c_{13}^5, \,
 2c_{12}^4= c_{21}^4\\

\mbox{new base for }\mathcal{Z}_{38}:
f_1=e_1, f_2=e_4, f_3=e_2, f_4=e_3, f_5=e_5 \\
\end{array}
\right\}$
 \\
\hline

$\mathcal{Z}_{40}$  &$\not \to$& 
$  \begin{array}{l}

\mathcal{Z}_{34}
\end{array} $
& ${\mathcal R}=
\left\{
\begin{array}{l}
 
 A_1^2 \subseteq A_2, \, A_1A_3+A_3A_1+A_2^2 \subseteq A_4, \\ 
 A_2A_3+A_3A_2 \subseteq A_5, \, A_5A_1+A_1A_5=0,\\

4 c_{14}^5 = c_{41}^5, \, 
3c_{23}^5 = 2 c_{32}^5, \,
3c_{13}^4= c_{31}^4, \,
c_{21}^3 = 2 c_{12}^3

\end{array}
\right\}$
 \\
\hline



\end{longtable}

\end{proof}


\begin{thebibliography}{99}




 
 
 
\bibitem{adashev}
Adashev J., Camacho L., Gomez-Vidal S., Karimjanov I., 
    Naturally graded Zinbiel algebras with nilindex $n-3,$ 
    Linear Algebra and its Applications, 443 (2014), 86--104.
 



\bibitem{comdend}
Aguiar M., 
    Pre-Poisson algebras, 
    Letters in Mathematical Physics, 54 (2000), 263--277.

 
 
\bibitem{Aloulou}
Aloulou W., Arnal D., Chatbouri R., 
    Algèbre pré-Gerstenhaber à homotopie près. (French),
    Journal of Pure Applied Algebra, 221 (2017), 11, 2666--2688.

\bibitem{ale3}
Alvarez M.A., 
    Degenerations of $8$-dimensional $2$-step nilpotent Lie algebras,
    Algebras and Representation Theory, 24 (2021), 5, 1231--1243.
 

 
\bibitem{aleis2}
Alvarez M.A., Hern\'{a}ndez I., 
    Varieties of Nilpotent Lie Superalgebras of dimension $\leq 5$,
    Forum Mathematicum, 32 (2020), 3,  641--661.


\bibitem{alesl}
  Alvarez M.A.,  Kaygorodov I.,  
  The algebraic and geometric classification of nilpotent weakly associative and symmetric Leibniz algebras,  
 Journal of  Algebra, 588  (2021),  278--314.

\bibitem{aleis}
 Alvarez M.A.,  Vera S., On rigid $3$-dimensional Hom-Lie algebras, 
Journal of Algebra,  588  (2021), 166--188.


 


\bibitem{abp} 
Artemovych O., Blackmore D., Prykarpatski A., 
    Non-associative structures of commutative algebras related with quadratic Poisson brackets,
    European Journal of Mathematics, 6 (2020), 1, 208--231.

 
\bibitem{bsaid}
Benayadi S., 
    On representations of symmetric Leibniz algebras, 
    Glasgow Mathematical Journal, 62 (2020),  S1, S99--S107.
 

  
 
\bibitem{bredo}
Bremner M., Dotsenko V., 
    Classification of regular parametrized one-relation operads, 
    Canadian Journal of  Mathematics, 69 (2017), 5, 992--1035.


\bibitem{cam13}
Camacho L.,  Ca\~nete E., G\'omez-Vidal S., Omirov B., 
    $p$-filiform Zinbiel algebras, 
    Linear Algebra and its Applications, 438 (2013), 7, 2958--2972. 
 

 
\bibitem{cam20} 
Camacho L., Karimjanov I., Kaygorodov I., Khudoyberdiyev  A.,
    Central extensions of filiform Zinbiel algebras,
    Linear and Multilinear Algebra,     70  (2022), 2, 1479--1495. 


\bibitem{Chapoton21}
 Chapoton F., 
    Zinbiel algebras and multiple zeta values,
    arXiv:2109.00241 



\bibitem{chouhy}
Chouhy S.,
    On geometric degenerations and Gerstenhaber formal deformations,
    Bulletin of the London Mathematical Society, 51 (2019),  5, 787--797.
    
\bibitem{cibils}  
Cibils C., 
    $2$-nilpotent and rigid finite-dimensional algebras,
    Journal of the London Mathematical Society (2), 36 (1987), 2, 211--218. 
    
    
\bibitem{degr3}
Cicalò S., De Graaf W.,   Schneider C.,
    Six-dimensional nilpotent Lie algebras,
    Linear Algebra and its Applications, 436 (2012), 1, 163--189.

 
\bibitem{rack} 
Covez S.,  Farinati M.,  Lebed V.,  Manchon D.,
    Bialgebraic approach to rack cohomology, 
    arXiv:1905.02754


 

\bibitem{degr2}
De Graaf W., 
    Classification of 6-dimensional nilpotent Lie algebras over fields of characteristic not $2$, 
    Journal of Algebra, 309  (2007), 2, 640--653.



\bibitem{tortnew1}
Diehl J.,  Ebrahimi-Fard K.,  Tapia N.,
    Time warping invariants of multidimensional time series,
    Acta Applicandae Mathematicae, 170 (2020), 265--290. 


\bibitem{tortnew2}
Diehl J.,  Lyons T.,  Preis R.,  Reizenstein J.,
    Areas of areas generate the shuffle algebra,  arXiv:2002.02338




\bibitem{dok}
Dokas I., 
Zinbiel algebras and commutative algebras with divided powers,
Glasgow Mathematical Journal, 52 (2010), 2, 303--313.


\bibitem{dzhuma}
Dzhumadildaev A., 
    Zinbiel algebras under $q$-commutators, 
    Journal of Mathematical Sciences (New York), 144 (2007), 2, 3909--3925.

\bibitem{dzhuma5}
Dzhumadildaev A., Tulenbaev K., 
    Nilpotency of Zinbiel algebras, 
    Journal of Dynamical and Control Systems, 11 (2005), 2, 195--213.


\bibitem{dzhuma19}
Dzhumadildaev A., Ismailov N., Mashurov F.,
    On the speciality of Tortkara algebras,
    Journal of Algebra, 540 (2019), 1--19.

 
 
 


  \bibitem{fkkv22}
Fern\'andez Ouaridi A., Kaygorodov I., Khrypchenko M.,  Volkov Yu., 
    Degenerations of nilpotent  algebras, 
    Journal of Pure and Applied Algebra,   226 (2022),  3, 106850.


       
  \bibitem{gabriel}
Gabriel P.,
Finite representation type is open,
Proceedings of the International Conference on Representations of Algebras (Carleton Univ., Ottawa, Ont., 1974), pp. 132--155.
 
 
 
\bibitem{matchzin}
 Gao X., Guo L.,  Zhang Yi.,
Commutative matching Rota-Baxter operators, shuffle products with decorations and matching Zinbiel algebras, 
Journal of Algebra,   586 (2021), 402--432.

  \bibitem{ger63}
Gerstenhaber M.,
    On the deformation of rings and algebras,
    Annals of Mathematics (2), 79 (1964), 59--103.


 
 
 
\bibitem{gorb93} 
Gorbatsevich V., 
    Anticommutative finite-dimensional algebras of the first three levels of complexity, 
    St. Petersburg Mathematical Journal, 5 (1994), 3, 505--521.

\bibitem{GRH}
Grunewald F.,  O'Halloran J.,
    Varieties of nilpotent Lie algebras of dimension less than six,
    Journal of Algebra, 112 (1988), 2, 315--325.

\bibitem{GRH2}
Grunewald F., O'Halloran J.,
    A Characterization of orbit closure and applications,
    Journal of Algebra, 116 (1988), 1, 163--175.

  
 
\bibitem{hac16}
Hegazi A., Abdelwahab H., Calderón Martín A.,
    The classification of $n$-dimensional non-Lie Malcev algebras with $(n-4)$-dimensional annihilator, 
    Linear Algebra and its Applications, 505 (2016), 32--56.


\bibitem{ikp20}
 Ignatyev M.,  Kaygorodov I., Popov Yu., 
  The geometric classification of $2$-step nilpotent algebras   and applications,   Revista Matemática Complutense, 2021, DOI:10.1007/s13163-021-00411-0 


\bibitem{ip21}
Ikonicoff S.,  Pacaud Lemay J.-S.,
 Cartesian Differential Comonads and New Models of Cartesian Differential Categories,
  arXiv:2108.04304
 
  
\bibitem{centr3zinb}
  Kaygorodov, I.,  Alvarez  M.A.,  Castilho de Mello T., Central extensions of 3-dimensional Zinbiel algebras,  
  Ricerche di Matematica,
   2021, DOI: 10.1007/s11587-021-00604-1 

 
 

\bibitem{kkp20}
Kaygorodov I., Khrypchenko M., Popov Yu., 
The algebraic and geometric classification of nilpotent terminal algebras, Journal of Pure and Applied Algebra,  225 (2021), 6, 106625.



\bibitem{klp20} Kaygorodov I., Lopes S., P\'{a}ez-Guill\'{a}n P.,   
Non-associative central extensions of null-filiform associative algebras, 
Journal of Algebra,   560  (2020),   1190--1210.


 
 


\bibitem{kppv}
Kaygorodov I.,  Popov Yu., Pozhidaev A., Volkov Yu.,
    Degenerations of Zinbiel and nilpotent Leibniz algebras,
    Linear and Multilinear Algebra,   66 (2018), 4, 704--716.
[Corrigendum to  "Degenerations of  Zinbiel and nilpotent Leibniz  algebras", 
Linear and Multilinear Algebra, 70 (2022), no. 5, 993–995.]


\bibitem{Kawski}
Kawski M.,
Chronological algebras: combinatorics and control,
Journal of Mathematical Sciences (New York), 103 (2001), 6, 725--744.

 
 \bibitem{pasha}
    Kolesnikov P., Commutator algebras of pre-commutative algebras,  
    Matematicheskii Zhurnal, 16 (2016), 2, 56--70.
 

\bibitem{loday}
Loday J.-L., 
Cup-product for Leibniz cohomology and dual Leibniz algebras, 
Mathematica Scandinavica, 77 (1995), 2, 189--196.

 
 
\bibitem{34}  Loday J.-L., 
On the algebra of quasi-shuffles, 
Manuscripta mathematica, 123 (2007), 79--93.
 
  

\bibitem{mukh}
Mukherjee G.,  Saha R., 
Cup-product for equivariant Leibniz cohomology and Zinbiel algebras, Algebra Colloquium, 26 (2019), 2, 271--284.


\bibitem{anau}
Naurazbekova A., 
On the structure of free dual Leibniz algebras, 
Eurasian Mathematical Journal, 10 (2019), 3,  40--47.

\bibitem{ualbay}
Naurazbekova A., Umirbaev U., 
Identities of dual Leibniz algebras, 
TWMS  Journal of Pure and Applied Mathematics, 1 (2010),  1, 86--91.


 


\bibitem{S90}
Seeley C., 
    Degenerations of 6-dimensional nilpotent Lie algebras over $\mathbb{C}$, 
    Communications in Algebra, 18 (1990), 3493--3505.


\bibitem{shaf}
    Shafarevich I., 
    Deformations of commutative algebras of class $2,$ Leningrad Mathematical Journal, 2 (1991), 6, 1335--1351.

\bibitem{ss78}
Skjelbred T., Sund T.,
    Sur la classification des algebres de Lie nilpotentes,
    C. R. Acad. Sci. Paris Ser. A-B, 286 (1978), 5,  A241--A242.

  

\bibitem{wolf1}
Volkov Yu., 
    Anticommutative Engel algebras of the first five levels, 
    Linear and Multilinear Algebra, 70 (2022), 1,  148--175.

\bibitem{wolf2}
Volkov Yu., 
    $n$-ary algebras of the first level,
     Mediterranean Journal of Mathematics, 19 (2022), 1, Paper: 2.
 
 

 

 


   
 

 
\end{thebibliography}
\end{document}